\documentclass[12pt]{article}
\usepackage{amsmath}
\usepackage{amssymb}
\usepackage{mathrsfs}
\usepackage{color}
\usepackage{array}
\usepackage[T1]{fontenc}
\usepackage[latin1]{inputenc}
\numberwithin{equation}{section}
\usepackage{datetime}
\usepackage[linkcolor=black,urlcolor=black,colorlinks=true]{hyperref}
\usepackage{enumerate}
\usepackage[pagewise]{lineno}
\def \beq {\begin{equation}}
\def \eeq {\end{equation}}
\def \ba {\begin{array}}
\def \ea {\end{array}}
\def \bs {\bigskip}

\def \ecart {\noalign{\medskip}}
\newenvironment{itemize*}
{\begin{itemize}
\setlength{\itemsep}{0.pt}
\setlength{\parskip}{0.pt}}
{\end{itemize}}
\def \dis {\displaystyle}

\renewcommand{\r}{\mathop{\rightarrow}}
\def\r{\rightarrow}
\renewcommand{\r}{\mathop{\rightarrow}}
\newcommand{\fdem}{\hspace*{\fill}~$\Box$\par\endtrivlist\unskip}

\def \div {\mbox{\rm div}\,}
\def \al {\alpha}

\def \ep {\varepsilon}

\def \la {\lambda}

\def \ph {\varphi}

\def \si {\sigma}

\newcommand{\N}{\mathbb{N}}     
\newcommand{\Z}{\mathbb{Z}}
\newcommand{\R}{\mathbb{R}}

\newcommand{\Q}{\mathbb{Q}}

\def \cD {\mathscr{D}}
\def \cE {\mathscr{E}}

\def \cI {\mathscr{I}}

\def \cM {\mathscr{M}}

\def \sfA {\mathsf{A}}

\def \sfC {\mathsf{C}}

\newenvironment{proof}[1]{\textit{Proof#1.\,}}{\fdem}

\newtheorem{atheo}{Theorem}[section]
\newtheorem{acor}{Corollary}[section]
\newtheorem{alem}{Lemma}[section]
\newtheorem{arem}{Remark}[section]
\newtheorem{Aexa}{Exemple}[section]

\newtheorem{apro}[alem]{Proposition}

\title{Asymptotics of ODE's flow on the torus through a singleton condition and a perturbation result. Applications}
\author{Marc Briane \& Lo\"\i c Herv\'e
\\*[.1em]
\normalsize Univ Rennes, INSA Rennes,  CNRS, IRMAR - UMR 6625, F-35000 Rennes, France
\\
\normalsize mbriane@insa-rennes.fr \& loic.herve@insa-rennes.fr
}
\begin{document}
\maketitle
\begin{abstract}
This paper deals with the long time asymptotics of the flow $X$ solution to the autonomous vector-valued ODE: $X'(t,x)=b(X(t,x))$ for  $t\in\R$, with $X(0,x)=x$ a point of the torus $Y_d$. We assume that the vector field $b$ reads as the product $\rho\,\Phi$, where $\rho$ is a non negative regular function and $\Phi$ is a non vanishing regular vector field in $Y_d$. In this work, the singleton condition means that the rotation set $\sfC_b$ composed of the average values of $b$ with respect to the invariant probability measures for the flow~$X$ is a singleton $\{\zeta\}$. This combined with Liouville's theorem regarded as a divergence-curl lemma, first allows us to obtain the asymptotics of the flow $X$, when $b$ is a nonlinear current field. Then, we prove a general perturbation result assuming that $\rho$ is the uniform limit in $Y_d$ of a positive sequence $(\rho_n)_{n\in\N}$ satisfying for any $n\in\N$, $\rho\leq\rho_n$ and $\sfC_{\rho_n\Phi}$ is a singleton $\{\zeta_n\}$. It turns out that the limit set $\sfC_b$ either remains a singleton, or enlarges to the closed line segment $[0_{\R^d},\lim_n\zeta_n]$ of~$\R^d$. We provide various corollaries of this perturbation result involving or not the classical ergodic condition, according to the positivity or not of some harmonic means of~$\rho$. These results are illustrated by different examples which highlight the alternative satisfied by the rotation set~$\sfC_b$. Finally, we prove that the singleton condition allows us to homogenize in any dimension the linear transport equation induced by the oscillating velocity $b(x/\ep)$ beyond any ergodic condition satisfied by the flow $X$.
\end{abstract}
\par\bs\noindent
\vskip 10.cm
%%%%%%%%%%
%%%%%%%%%%
\tableofcontents
\section{Introduction}
In this paper we study the large time asymptotics of the solution $X(\cdot,x)$ for $x\in\R^d$, to the~ODE
\beq\label{bXi}
\left\{\ba{ll}
\dis {\partial X\over\partial t}(t,x)=b(X(t,x)), & t\in\R
\\ \ecart
X(0,x)=x,
\ea\right.
\eeq
where $b$ is a $C^1$-regular vector field defined in the torus $Y_d:=\R^d\setminus\Z^d$ (denoted by $b\in C^1_\sharp(Y_d)^d$) according to the $\Z^d$-periodicity \eqref{fper}. The solution $X(\cdot,x)$ is well-defined for any $x\in Y_d$ by virtue of \eqref{XxperY}. More precisely, we focus on the existence of the limit of $X(t,x)/t$ as $t\to\infty$ for $x\in Y_d$.
This question naturally arises in ergodic theory, since it involves the {\em dynamic flow} $X$ induced by ODE \eqref{bXi}, and the Borel measures $\mu$ on the torus $Y_d$ which are {\em invariant for the flow} $X$, {\em i.e.}
\beq\label{invmu}
\forall\,t\in\R,\ \forall\,\psi\in C^0_\sharp(Y_d),\quad\int_{Y_d}\psi\big(X(t,y)\big)\,d\mu(y)=\int_{Y_d}\psi(y)\,d\mu(y).
\eeq
A strengthened variant of the famous Birkhoff ergodic theorem \cite[Theorem~2, Section~1.8]{CFS} claims that if the flow is {\em uniquely ergodic}, {\em i.e.} there exists a unique probability measure $\mu$ on $Y_d$ which is invariant for the flow, then any function~$f\in C^0_\sharp(Y_d)$ satisfies
\beq\label{Ttuninv}
\forall\,x\in Y_d,\quad \lim_{t\to\infty}\left[{1\over t}\int_0^t f(X(s,x))\,ds\right]=\int_{Y_d}f(y)\,d\mu(y),
\eeq
and the converse actually holds true.
In the particular case where $f=b$, limit \eqref{Ttuninv} yields
\beq\label{asyXi}
\forall\,x\in Y_d,\quad \lim_{t\to\infty}{X(t,x)\over t}=\int_{Y_d}b(y)\,d\mu(y).
\eeq
The unique ergodicity condition is a rather restrictive condition on the flow~\eqref{bXi}.
Alternatively, define the set
\beq\label{Ib}
\cI_b:=\big\{\mu\in\cM_p(Y_d): \mu\mbox{ invariant for the flow }X\big\},
\eeq
where $\cM_p(Y_d)$ is the set of probability measures on $Y_d$, and the subset of $\cI_b$
\beq\label{Eb}
\cE_b:=\big\{\mu\in\cI_b: \mu\mbox{ ergodic for the flow }X\big\}.
\eeq
It is known that the ergodic measures for the flow $X$ are the extremal points of the convex set $\cI_b$ so that
\beq\label{IbEb}
\cI_b={\rm conv}(\cE_b).
\eeq
Also define for any vector field $b\in C^1_\sharp(Y_2)^d$ the two following non empty subsets of $\R^d$:
\begin{itemize}
\item The set of all the limit points of the sequences $\big(X(n,x)/n\big)_{n\geq 1}$ for $x\in Y_d$ (denoted by $\rho_{\rm p}(b)$ in \cite{MiZi})
\beq\label{Ab}
\sfA_b:=\bigcup_{x\in Y_d}\left[\,\bigcap_{n\geq 1}\overline{\left\{{X(k,x)\over k}:k\geq n\right\}}\right].
\eeq
\item The so-called Herman \cite{Her} rotation set
\beq\label{Cb}
\sfC_b:=\left\{\int_{Y_2}b(y)\,d\mu(y):\mu\in\cI_b\right\}={\rm conv}\left\{\int_{Y_2}b(y)\,d\mu(y):\mu\in \cE_b\right\},
\eeq
which is compact and convex.
\end{itemize}
An implicit consequence of \cite[Theorem~2.4, Remark~2.5, Corollary~2.6]{MiZi} shows that
\beq\label{ACb}
\sfA_b\subset \sfC_b={\rm conv}(\sfA_b)\quad\mbox{and}\quad \# \sfA_b=1\Leftrightarrow \# \sfC_b=1,
\eeq
Note that by definition~\eqref{Ab} the equivalence of \eqref{ACb} can be written for any $\zeta\in \R^d$,
\beq\label{Cbsindis}
\sfC_b=\{\zeta\}\quad\Leftrightarrow\quad\forall\,x\in Y_d,\;\; \lim_{n\to\infty}{X(n,x)\over n}=\zeta\quad\Leftrightarrow\quad
\forall\,x\in Y_d,\;\; \lim_{t\to\infty}{X(t,x)\over t}=\zeta.
\eeq
\par\bigskip\noindent
In the sequel the ``singleton condition'' means that Herman's rotation set $\sfC_b$ is a singleton $\{\zeta\}$.
Proposition~\ref{pro.Cb} below provides an alternative proof of \eqref{Cbsindis}.
Then, the ``singleton approach'' consists in establishing sufficient conditions on the vector field $b$ to ensure the singleton condition.
The aim of this paper is to exploit this approach either to get the asymptotics of the flow $X$ under suitable vector fields $b$, or in the less favorable cases to determine Herman's rotation set $\sfC_b$ as a closed line segment of $\R^d$.
\par
First of all, revisiting Liouville's theorem for invariant probability measures (see, {\em e.g.}, \cite[Theorem~1, Section~2.2]{CFS}) as a divergence-curl lemma (see Proposition~\ref{pro.divcurl}) we obtain (see Proposition~\ref{pro.FDv}) a rather surprising null asymptotics of the flow $X$ associated with a nonlinear current field of type $b=F(\cdot,\nabla v)$, where $F(x,\xi)$ is a vector-valued function in $C^1_\sharp(Y_d;C^1(\R^d))^d$ which is strictly monotonic with respect to variable $\xi$, and $v$ is a scalar potential in $C^2_\sharp(Y_d)$.
\par
Actually, except the one-dimensional case where $b$ is parallel to a fixed direction so that the flow can be computed explicitly (see Example~\ref{exa.1}), there are very few examples of vector field $b$ for which the asymptotics of the flow \eqref{bXi} is completely known in dimension $d\geq 2$.
There are at least two cases:
\begin{itemize}
\item If $b$ is a non vanishing regular field in dimension two, Peirone~\cite[Theorem~3.1]{Pei} has proved that the asymptotics of the flow $X(\cdot,x)$ does exist at each point of $x\in Y_2$. Moreover, the singleton condition is satisfied under the extra ergodic condition that for any $x\in Y_2$, the flow $X(\cdot,x)$ is not periodic in $Y_2$ according to \eqref{XperYd}.
\item Under the global rectification condition $\nabla\Psi\,b=\zeta$ in $Y_d$, where $\Psi$ is a $C^2$-diffeomorphism on the torus $Y_d$ and $\zeta$ is a non zero constant vector in $\R^d$, the set $\sfC_b$ is the singleton $\big\{(\int_{Y_d}\nabla\Psi)^{-1}\zeta\big\}$ in any dimension (see \cite[Corollary~4.1]{Bri2} and Remark~\ref{rem.sta}).
\end{itemize}
\par
In the two former situations the vector field $b$ does not vanish. In order to extend these two results among others to a vanishing vector field $b$, a natural question is to know if the singleton condition is stable under a uniform non vanishing perturbation $b_n$ of $b$ for $n\in\N$. We provide a partial answer to this question with the main result Theorem~\ref{thm.Cbsin} of the paper. Restricting ourselves to a perturbation $b_n=\rho_n\,\Phi$, where $(\rho_n)_{n\in\N}$ is a sequence of positive functions in $C^1_\sharp(Y_d)$ converging uniformly in $Y_d$ to some regular function $\rho\leq\rho_n$ and where $\Phi$ is a fixed vector field, we prove that if $\sfC_{b_n}=\{\zeta_n\}$ for any $n\in\N$, then the sequence $(\zeta_n)_{n\in\N}$ converges to some $\zeta\in\R^d$. Moreover, we get that the limit set $\sfC_b$ is the singleton $\{\zeta\}$ if $\rho$ is positive, and that $\sfC_b$ is the closed line segment $[0_{\R^d},\zeta]$ if $\rho$ is only non negative.
In Theorem~\ref{thm.Cbsin} it is essential that the vector field $\Phi$ in $b_n=\rho_n\,\Phi$ is independent of $n$, otherwise the perturbation result does not hold in general  (see Remark~\ref{rem.Phindep} and Example~\ref{exa.1}). Moreover, the two-dimensional Example~\ref{exa.2} shows that the sequence of singletons $(\sfC_{b_n})_{n\in\N}$ may be actually enlarged to the limit closed line segment $\sfC_b=[0_{\R^d},\zeta]$ with $\zeta\neq 0$.
From the perturbation result we first deduce (see Corollary~\ref{cor.AbCb}) that for a fixed vector field $\Phi\in C^1_\sharp(Y_d)^d$, the set of the positive functions $\rho\in C^1_\sharp(Y_d)$ with $\#\sfC_{\rho\Phi}=1$ is closed for the uniform convergence in $C^1_\sharp(Y_d)^d$. This result does not extend to the larger set composed of the non negative functions $\rho$ as shown in Remark~\ref{rem.AbCb}. 
Then, we prove various corollaries of Theorem~\ref{thm.Cbsin} in terms of the asymptotics of the ODE's flow \eqref{bXi}:
\begin{itemize}
\item We show (see Corollary~\ref{cor.dist-orbite}) that the asymptotics of the flow $X(\cdot,x)$ associated with $b=\rho\,\Phi$ does exist at any point $x$ such that the orbit $X(\R,x)$ is far enough from the set $\{\rho=0\}$.
\item When the flow associated with the vector field $\Phi$ admits an invariant probability measure with a positive density $\sigma\in C^1_\sharp(Y_d)$ with respect to Lebesgue's measure, we prove (see Corollary~\ref{cor.siPhidiv0}) the following alternative:
\begin{itemize}
\item $\sfC_{\rho\Phi}=\{0_{\R^d}\}$ if the harmonic mean of $\rho/\sigma$ is equal to $0$,
\item $\sfC_{\rho\Phi}$ is some closed line segment $[0_{\R^d},\zeta]$ of $\R^d$ if the harmonic mean of $\rho/\sigma$ is positive (see the two-dimensional Example~\ref{exa.2}).
\end{itemize}
\item As a by-product of Corollary~\ref{cor.siPhidiv0} we determine (see Corollary~\ref{cor.CbP}) the set $\sfC_{\rho\Phi}$ in dimension two when $\Phi$ is parallel to an orthogonal gradient satisfying an ergodic condition, extending Peirone's result \cite[Theorem~3.1]{Pei} (see Remark~\ref{rem.SP}) to the case where the vector field $b$ does vanish. Corollary~\ref{cor.CbP} is illustrated in Example~\ref{exa.4} by the case where $b$ is a two-dimensional electric field, while dimension three is shown to be quite different. Similarly, we extend (see Corollary~\ref{cor.Cbdet}) the non ergodic case of \cite[Corollary~4.1]{Bri2} to a vanishing vector field $b$ in any dimension (see Example~\ref{exa.3}).
\end{itemize}
It turns out that the singleton approach cannot be regarded exclusively as an ergodic approach. It may contain an ergodic condition as in Corollary~\ref{cor.CbP}. But it may be also independent of any ergodic condition as in Corollary~\ref{cor.Cbdet}. This does make this approach an original alternative to the classical ergodic approach.
\par
Finally, we apply the asymptotics of the flow \eqref{bXi} to the homogenization of the linear transport equation with an oscillating velocity 
\beq\label{traequ}
{\partial u_\ep\over\partial t}(t,x)+b\left({x\over\ep}\right)\cdot\nabla_x u_\ep(t,x)=0\quad\mbox{for }(t,x)\in[0,\infty)\times\R^d,
\eeq
where the vector field $b$ belongs to $C^1_\sharp(Y_d)^d$. Tartar \cite{Tar} and Amirat {\em et al} \cite{AHZ1,AHZ2,AHZ3} showed that in general the homogenization of equation~\eqref{traequ} leads to a nonlocal limit problem. Here, we focus on the cases where the homogenized equation remains a linear transport equation with some average velocity $\langle b\rangle$ of the vector field $b$. In this perspective, the following works may be quoted:
First, assuming that $b$ is divergence free and the flow associated with $b$ is ergodic, Brenier \cite{Bre} proved the convergence of the solution $u_\ep$ in any dimension. This result was extended by Golse \cite[Theorem~8]{Gol1} (see also \cite{Gol2}) for a more general velocity $b(x,x/\ep)$ with ${\rm div}_yb(x,\cdot)=0$, assuming the ergodicity of the flows associated with the vector fields $b(x,\cdot)$. Hou and Xin~\cite{HoXi} performed the homogenization of \eqref{traequ} in dimension two with an oscillating initial condition $u_\ep(0,x)=u^0(x,x/\ep)$, assuming that $b$ is a non vanishing divergence free vector field in $\R^2$ and that the flow~\eqref{bXi} is ergodic. To this end, they used a two-scale convergence approach combined with Kolmogorov's theorem \cite{Kol} involving some rotation number.
This two-scale approach based on the divergence free condition was extended by Jabin and Tzaveras \cite{JaTz} using a kinetic decomposition in the two-scale procedure.
Moreover, Tassa~\cite{Tas} extended the two-dimensional homogenization result of \cite{HoXi}, assuming that the flow $X$ associated with $b$ has an invariant probability measure with a positive regular density with respect to Lebesgue's measure. More generally, Peirone~\cite{Pei} proved the convergence of the solution $u_\ep$ to equation \eqref{traequ} in dimension two, under the sole assumption that $b$ does not vanish in~$Y_2$. More recently, the first author proved in \cite[Corollary~4.4]{Bri2} (see also~\cite{Bri3} for an extension to the non periodic case) the homogenization of \eqref{traequ}, replacing the classical ergodic condition by the rectification condition $\nabla\Psi\,b=\zeta$ for some $C^2$-diffeomorphism $\Psi$ on~$Y_d$, and a non null constant vector $\zeta\in\R^d$.
\par
In the present case, extending the previous results in the case where the initial condition $u_\ep(0,\cdot)$ does not oscillate, we prove a new result (see Theorem~\ref{thm.hom}) on the homogenization of the linear transport equation \eqref{traequ} in any dimension, only assuming the singleton condition
(without therefore assuming that the velocity of transport equation~\eqref{traequ} is divergence free).
The results of Section~\ref{s.ergres} and Section~\ref{s.newres} provide various and rather general situations where the singleton condition applies:
\begin{itemize*}
\item the case of the nonlinear current field in Proposition~\ref{pro.FDv},
\item some of the cases of Corollary~\ref{cor.siPhidiv0} and Corollary~\ref{cor.CbP},
\item Corollary~\ref{cor.Cbdet} illustrated by Example~\ref{exa.3},
\item Example~\ref{exa.2} when $\alpha\geq 1$,
\item the two-dimensional conductivity case of Example~\ref{exa.4} under the ergodic condition~\eqref{ergA*la}.
\end{itemize*}
\par
The paper is organized as follows. In Section~\ref{s.ergres} we revisit the singleton condition and the Liouville theorem, from which we deduce the asymptotics of the flow~\eqref{bXi} when $b$ is a current field. In Section~\ref{s.newres} we establish the perturbation Theorem~\ref{thm.Cbsin} which is the main result of the paper, and we derive various corollaries on the set $\sfC_b$ and on the asymptotics of the flow~$X$.
Section~\ref{s.exa} presents four examples which illustrate the results of Section~\ref{s.ergres} and Section~\ref{s.newres}.
Section~\ref{s.hom} is devoted to the homogenization of the transport equation~\eqref{traequ} in connection with the asymptotics of the flow.
\par\noindent
\subsection*{Notation}
\begin{itemize}
\item $(e_1,\dots,e_d)$ denotes the canonical basis of $\R^d$, and $0_{\R^d}$ denotes the null vector of $\R^d$.
\item $|\cdot|$ denotes the euclidian norm in $\R^N$.
\item $Y_d$ for $d\geq 1$, denotes the $d$-dimensional torus $\R^d/\Z^d$, which is identified to the cube $[0,1)^d$ in $\R^d$.
\item $C^k_c(\R^d)$ for $k\in\N\cup\{\infty\}$, denotes the space of the real-valued functions in $C^k(\R^d)$ with compact support.
\item $C^k_\sharp(Y_d)$ for $k\in\N\cup\{\infty\}$, denotes the space of the real-valued functions $f\in C^k(\R^d)$ which are $\Z^d$-periodic, {\em i.e.}
\beq\label{fper}
\forall\,\kappa\in\Z^d,\ \forall\,x\in \R^d,\quad f(x+\kappa)=f(x).
\eeq
\item $L^p_\sharp(Y_d)$ for $p\geq 1$, denotes the space of the real-valued functions in $L^p_{\rm loc}(\R^d)$ which are $\Z^d$-periodic.
\item  $\cM(\R_d)$, resp. $\cM(Y_d)$, denotes the space of the Radon measures on $\R^d$, resp. $Y_d$, and $\cM_p(Y_d)$ denotes the space of the probability measures on $Y_d$.
\item The notation $\cI_b$ in \eqref{Ib} will be used throughout the paper.
\item $\cD'(\R^d)$ denotes the space of the distributions on $\R^d$.
\item Let $a$ be a non negative function in $L^1_\sharp(Y_d)$, the arithmetic mean $\overline{a}$ and the harmonic mean $\underline{a}$ of $a$ are defined by
\[
\overline{a}:=\int_{Y_d}a(y)\,dy\quad\mbox{and}\quad \underline{a}:=\left(\int_{Y_d}{dy\over a(y)}\right)^{-1}.
\]
\end{itemize}
\subsection*{Definitions and recalls}
Let $b:\R^d\to\R^d$ be a vector-valued function in $C^1_\sharp(Y_d)^d$.
Consider the dynamical system
\beq\label{bX}
\left\{\ba{ll}
\dis {\partial X\over\partial t}(t,x)=b(X(t,x)), & t\in\R
\\ \ecart
X(0,x)=x\in\R^d.
\ea\right.
\eeq
The solution $X(\cdot,x)$ to \eqref{bX} which is known to be unique (see, {\em e.g.}, \cite[Section~17.4]{HSD}) induces the dynamic flow $X$ defined by
\beq\label{X}
\ba{lrll}
X: & \R\times \R^d & \to & \R^d
\\ \ecart
& (t,x) & \mapsto & X(t,x),
\ea
\eeq
which satisfies the semi-group property
\beq\label{sgroup}
\forall\,s,t\in\R,\ \forall\,x\in \R^d,\quad X(s+t,x)=X(s,X(t,x)).
\eeq
The flow $X$ is actually well defined in the torus $Y_d$, since
\beq\label{XxperY}
\forall\,t\in\R,\ \forall\,x\in \R^d,\ \forall\,\kappa\in\Z^d,\ \quad X(t,x+\kappa)=X(t,x)+\kappa.
\eeq
Property~\eqref{XxperY} follows immediately from the uniqueness of the solution $X(\cdot,x)$ to \eqref{bX} combined with the $\Z^d$-periodicity of $b$.
\par
For any $x\in Y_d$, the solution $X(\cdot,x)$ to \eqref{bX} is said to be {\em periodic in the torus} $Y_d$ if there exist $T>0$ and $\kappa\in\Z^d$ such that
\beq\label{XperYd}
\forall\,t\in\R,\quad X(t+T,x)=X(t,x)+\kappa.
\eeq
If $\kappa=0_{\R^d}$ the solution is said to be {\em periodic in $\R^d$}.
%%%%%%%%%%
\section{Some variants of classical ergodicity results}\label{s.ergres}
\subsection{The singleton result}\label{ss.appsin}
Equivalences \eqref{Cbsindis} have been obtained in \cite{MiZi} as a consequence of the so-called ergodic decomposition theorem. Here, we provide a simpler and more direct proof  of \eqref{Cbsindis}, which is based on Proposition~\ref{pro.invmeas} (in the Appendix) only involving the weak-$*$ compactness of $\cM_p(Y_d)$.
Also note that the uniform convergence result below is mentioned in \cite[Section~9]{Her} with no reference.
\begin{apro}\label{pro.Cb}
Let $b\in C^1_\sharp(Y_d)^d$. Then, the following equivalence holds for any $\zeta\in\R^d$,
\beq\label{Cbsin}
\ba{lll}
\sfC_b=\{\zeta\} & \Leftrightarrow & \dis \forall\,x\in Y_d,\;\;\lim_{t\to\infty}{X(t,x)\over t}=\zeta
\\ \ecart
& \Leftrightarrow & \dis {X(t,\cdot)\over t}\mbox{ converges uniformly as $t\to\infty$ to $\zeta$ on $Y_d$}.
\ea
\eeq
\end{apro}
\begin{proof}{ of Proposition~\ref{pro.Cb}}
First, assume that $\sfC_b=\{\zeta\}$. Assume by contradiction that the second right hand-side of \eqref{Cbsin} does not hold.
Then, there exists $\ep>0$ such that for any $n\in\N$, there exist a number $r_n>n$ and a point $y_n\in Y_d$ satisfying 
\beq\label{rnynzeep}
\left|\,{X(r_n,y_n)-y_n\over r_n}-\zeta\,\right|=\left|\,\frac{1}{r_n} \int_0^{r_n} b(X(s,y_n))\,ds-\zeta\,\right|\geq \ep.
\eeq
Now, let $\nu_n$ for $n\in\N$, be the probability measure on $Y_d$ defined by
\beq\label{nunf}
\int_{Y_d} f(y)\, d\nu_n(y) = \frac{1}{r_n} \int_0^{r_n} f(X(s,y_n))\,ds\quad\mbox{for }f\in C^0_\sharp(Y_d)^d.
\eeq
By virtue of Lemma~\ref{lem-extrac} there exists a subsequence $(\nu_{n_k})_{k\in\N}$ of $(\nu_n)_{n\in\N}$ which converges weakly~$*$ to some probability measure $\mu\in \cM_p(Y_d)$ which is invariant for the flow $X$.
Hence, passing to the limit as $n_k\to\infty$ both in \eqref{rnynzeep} and \eqref{nunf} with $f:=b$, we deduce from $\sfC_b=\{\zeta\}$ that
\[
|\zeta-\zeta|=\left|\,\int_{Y_d} b(y)\, d\mu(y)-\zeta\,\right|\geq \ep>0,
\]
which yields a contradiction.
\par\noindent
It is clear that the second right-hand side of \eqref{Cbsin} implies the first one.
\par
Finally, let us assume that the first right-hand side of \eqref{Cbsin} holds true with $\zeta$, and let us prove that $\sfC_b=\{\zeta\}$.
We thus have
\[
\forall\,x\in Y_d,\quad \lim_{t\to\infty}\left({1\over t}\int_0^t b(X(s,x))\,ds\right)=\zeta.
\]
Then, integrating over $Y_d$ the former equality with respect to any probability measure $\mu\in\cI_b$, then applying successively Lebesgue's dominated convergence theorem and Fubini's theorem, we get that
\[
\ba{ll}
\zeta & \dis =\lim_{t\to\infty}\int_{Y_d}\left({1\over t}\int_0^t b(X(s,x))\,ds\right)d\mu(x)
\\ \ecart
& \dis =\lim_{t\to\infty}{1\over t}\int_0^t\left(\int_{Y_d} b(X(s,x))\,d\mu(x)\right)ds
=\int_{Y_d} b(x)\,d\mu(x),
\ea
\]
which shows that $\sfC_b=\{\zeta\}$. This concludes the proof of \eqref{Cbsin}.
\end{proof}
\begin{arem}\label{rem.sta}
Stability of the singleton condition by a diffeomorphism on the torus.
\par\noindent
A mapping $\Psi\in C^1(\R^d)^d$ is said to be a $C^1$-diffeomorphism on $Y_d$ if $\Psi$ satisfies the following conditions:
\begin{itemize}
\item $\det(\nabla\Psi(x))\neq 0$ for any $x\in\R^d$,
\item there exist a matrix $A\in\Z^{d\times d}$ with $|\det(A)|=1$, and a mapping $\Psi_\sharp\in C^1_\sharp(Y_d)^d$ such that
\beq\label{PsiAPsid}
\forall\, x\in\R^d, \quad \Psi(x)=Ax+\Psi_\sharp(x).
\eeq
\end{itemize}
Note that the invertibility of $A$ and the $\Z^d$-periodicity of $\Psi_\sharp$ in \eqref{PsiAPsid} imply that $\Psi$ is a proper function ({\em i.e.}, the inverse image by the function of any compact set in $\R^d$ is a compact set). Hence, by virtue of Hadamard-Caccioppoli's theorem \cite{Cac} (also called Hadamard-L\'evy's theorem) the mapping $\Phi$ is actually a $C^1$-diffeomorphism on $\R^d$.
Also note that due to $A^{-1}\in\Z^{d\times d}$, we have
\[
\forall\,\kappa\in\Z^d,\ \forall\,x\in Y_d,\quad
\left\{\ba{ll}
\Psi(x+\kappa)-\Psi(x)=A\kappa & \in\Z^d
\\ \ecart
\Psi^{-1}(x+\kappa)-\Psi^{-1}(x)=A^{-1}\kappa & \in\Z^d,
\ea\right.
\]
hence $\Psi$ well defines an isomorphism on the torus.
\par
Now, let $b$ be a vector field in $C^1_\sharp(Y_d)^d$ and let $\Psi$ be a $C^2$-diffeomorphism on $Y_d$. Define the flow $\tilde{X}$ obtained from $\Psi$ by
\beq\label{tX}
\tilde{X}(t,x):=\Psi\big(X(t,\Psi^{-1}(x))\big)\quad\mbox{for }(t,x)\in \R\times Y_d.
\eeq
Using the chain rule it is easy to check that the mapping $\tilde{X}$ is the flow associated with the vector field $\tilde{b}\in C^1_\sharp(Y_d)^d$ defined by
\beq\label{tb}
\tilde{b}(x)=\nabla\Psi(\Psi^{-1}(x))\,b(\Psi^{-1}(x))\quad\mbox{for }x\in Y_d.
\eeq
Combining \eqref{PsiAPsid} and \eqref{tX} we clearly have
\[
\forall\,x\in Y_d,\quad\lim_{t\to\infty}{\tilde{X}(t,x)\over t}\mbox{ exists}\;\Leftrightarrow\;\;\lim_{t\to\infty}{X(t,\Psi^{-1}(x))\over t}\mbox{ exists},
\]
and in the case of existence of the limit for a given $x\in Y_d$, we get the equality
\[
\lim_{t\to\infty}{\tilde{X}(t,x)\over t}=A\left(\lim_{t\to\infty}{X(t,\Psi^{-1}(x))\over t}\right).
\]
This combined with equivalence \eqref{Cbsin} implies that
\beq\label{CbCtb}
\# \sfC_{\tilde{b}}=1\;\Leftrightarrow\; \# \sfC_b=1,
\eeq
and in this case we obtain the equality $\sfC_{\tilde{b}}=A\,\sfC_b$.
Therefore, the singleton condition is stable by any $C^2$-diffeomorphism on $Y_d$.
\par
In particular, such a diffeomorphism has been used by Tassa \cite{Tas} in dimension two, assuming that the first coordinate $b_1$ of $b$ does not vanish in $Y_2$ and that there exists an invariant probability measure for the flow associated with $b$ having a density $\sigma\in C^1_\sharp(Y_2)$ with respect to Lebesgue's measure. In this case a variant \cite[Theorem~2.3]{Tas} of Kolmogorov's theorem \cite{Kol} (which holds under the weaker assumption that $b$ is non vanishing) provides a diffeomorphism on $Y_2$ which rectifies the vector field $b$ to a vector field $\tilde{b}=a\,\xi$ with a positive function $a\in C^1_\sharp(Y_2)$ and a fixed direction $\xi\in\R^d$. Under the additional ergodic assumption that the coordinates of $\xi$ are rationally independent, the singleton condition is shown to be satisfied \cite[Theorem~4.2]{Tas}. Corollary~\ref{cor.CbP} below provides a more general result in dimension two with a vanishing vector field $b$, without using  a rectification of the field $b$.
\end{arem}
\subsection{A divergence-curl result}
Liouville's theorem provides a criterium for a probability measure on a smooth compact manifold in $\R^d$ (see, {\em e.g.}, \cite[Theorem 1, Section~2.2]{CFS}) to be invariant for the flow.
The next result revisits this theorem in $\cM_p(Y_d)$ in association with a divergence-curl result on the torus.
\begin{apro}\label{pro.divcurl}
Let $b\in C^1_\sharp(Y_d)^d$ and let $\mu\in\cM_p(Y_d)$. Define the Borel measure $\tilde{\mu}$ on~$\R^d$ by
\beq\label{tmu}
\int_{\R^d}\ph(x)\,d\tilde{\mu}(x):=\int_{Y_d} \ph_\sharp(y)\,d\mu(y)\quad\mbox{where}\quad \ph_\sharp(\cdot):=\sum_{\kappa\in\Z^d}\ph(\cdot+\kappa)\quad\mbox{for }\ph\in C^0_c(\R^d).
\eeq
Then, the three following assertions are equivalent:
\begin{enumerate}[$(i)$]
\item $\mu$ is invariant for the flow $X$, {\em i.e.} \eqref{invmu} holds,
\item $\tilde{\mu}\,b$ is divergence free in $\R^d$, {\em i.e.}
\beq\label{dtmub=0}
\div(\tilde{\mu}\,b)=0\quad\mbox{in }\cD'(\R^d),
\eeq
\item $\mu\,b$ is divergence free in $Y_d$, {\em i.e.}
\beq\label{dmub=0}
\forall\,\psi\in C^1_\sharp(Y_d),\quad \int_{Y_d} b(y)\cdot\nabla\psi(y)\,d\mu(y)=0.
\eeq
\end{enumerate}
\end{apro}
\begin{proof}{ of Proposition~\ref{pro.divcurl}}
\hfill
\par\smallskip\noindent
{\it Proof of $(i)\Rightarrow(ii)$.}
Assume that $\mu$ is invariant for the flow, {\em i.e.} \eqref{invmu}.
Let $\ph\in C^1_c(\R^d)$. Since by \eqref{XxperY} we have for any $t\in\R$ and $y\in \R^d$,
\beq\label{phdph}
\big[\ph(X(t,\cdot))\big]_\sharp(y)=\sum_{\kappa\in\Z^d}\ph(X(t,y+\kappa))=\sum_{\kappa\in\Z^d}\ph(X(t,y)+\kappa)=\ph_\sharp(X(t,y)),
\eeq
it follows from \eqref{tmu} and the invariance of $\mu$ that
\[
\ba{ll}
\forall\,t\in\R, & \dis \int_{\R^d} \ph(X(t,x))\,d\tilde{\mu}(x)=\int_{Y_d} \big[\ph(X(t,\cdot))\big]_\sharp(y)\,d\mu(y)=\int_{Y_d} \ph_\sharp(X(t,y))\,d\mu(y)=
\\*[1.em]
& \dis \int_{Y_d} \ph_\sharp(y)\,d\mu(y)=\int_{\R^d} \ph(x)\,d\tilde{\mu}(x).
\ea
\]
Taking the derivative of the former expression with respect to $t$, we get that
\[
\forall\,t\in\R,\quad \int_{\R^d} b(X(t,x))\cdot\nabla\ph(X(t,x))\,d\tilde{\mu}(x)=0,
\]
which at $t=0$ yields
\beq\label{dtmub=02}
\forall\,\ph\in C^1_c(\R^d),\quad \int_{\R^d}b(x)\cdot\nabla\ph(x)\,d\tilde{\mu}(x)=0,
\eeq
namely the variational formulation of the distributional equation \eqref{dtmub=0}.
\par\medskip\noindent
{\it Proof of $(ii)\Rightarrow(i)$.}
Conversely, assume that equation \eqref{dtmub=0} holds true, and let us prove that $\mu$ is invariant for the flow $X$. Let $\ph\in C^1_c(\R^d)$ and define the function $\phi\in C^1(\R\times\R^d)$ by $\phi(t,x):=\ph(X(t,x))$.
By the semi-group property \eqref{sgroup} we have for any $s,t\in\R$ and $x\in\R^d$,
\[
\ba{ll}
\dis {\partial\over\partial s}\big(\phi(s+t,X(-s,x))\big) & \dis ={\partial\over\partial s}\big(\phi(t,x)\big)=0
\\ \ecart
& \dis ={\partial\phi\over\partial s}(s+t,X(-s,x))-b(X(-s,x))\cdot\nabla_x\phi(s+t,X(-s,x)),
\ea
\]
which at $s=0$ gives the classical transport equation
\beq\label{phib}
\forall\,t\in\R,\ \forall\,x\in Y_d,\quad{\partial\phi\over\partial t}(t,x)=b(x)\cdot\nabla_x\phi(t,x).
\eeq
Hence, since $\ph(X(t,\cdot))$ is in $C^1(\R^d)$ and has a compact support independent of $t$ when $t$ lies in a compact set of $\R$, we deduce from \eqref{phib} and \eqref{dtmub=0} that
\[
\forall\,t\in\R,\quad{d\over\,dt}\left(\int_{\R^d} \ph(X(t,x))\,d\tilde{\mu}(x)\right)=\int_{\R^d} b(x)\cdot\nabla_x\big(\ph(X(t,x))\big)\,d\tilde{\mu}(x)=0,
\]
or equivalently,
\[
\forall\,t\in\R,\quad \int_{\R^d} \ph(X(t,x))\,d\tilde{\mu}(x)=\int_{\R^d} \ph(x)\,d\tilde{\mu}(x).
\]
On the other hand, we have the following result.
\begin{alem}[\cite{Bri1}, Lemma~3.5]\label{lem.phd}
For any smooth function $\psi\in C^\infty_\sharp(Y_d)$ defined in $Y_d$, there exists a smooth function $\ph\in C^\infty_c(\R^d)$ with compact support in $\R^d$ such that $\psi=\ph_\sharp$.
\end{alem}
Hence, using relation \eqref{phdph} and definition \eqref{tmu} we get that for any $\psi\in C^\infty_\sharp(Y)$,
\[
\ba{ll}
\forall\,t\in\R, & \dis \int_{Y_d} \psi(X(t,y))\,d\mu(y)=\int_{Y_d} \ph_\sharp(X(t,y))\,d\mu(y)=\int_{\R^d} [\ph\circ X(t,\cdot)]_\sharp(y)\,d\mu(y)=
\\*[1.em]
&  \dis =\int_{\R^d} \ph(X(t,x))\,d\tilde{\mu}(x)=\int_{\R^d} \ph(x)\,d\tilde{\mu}(x)=\int_{Y_d} \ph_\sharp(y)\,d\mu(y)=\int_{Y_d} \psi(y)\,d\mu(y),
\ea
\]
which shows that $\mu$ is invariant for the flow $X$. We have just proved the equivalence between the invariance of $\mu$ for the flow and the distributional equation \eqref{dtmub=0} satisfied by $\tilde{\mu}$.
\par\medskip\noindent
{\it Proof of $(ii)\Leftrightarrow(iii)$.}
The equivalence between \eqref{dtmub=0}, or equivalently \eqref{dtmub=02}, and \eqref{dmub=0} is a straightforward consequence of the following relation (which is deduced from $[b\cdot\nabla\ph]_\sharp=b\cdot\nabla\ph_\sharp$ and \eqref{tmu})
\[
\forall\,\ph\in C^1_c(\R^d),\quad \int_{\R^d}b(x)\cdot\nabla\ph(x)\,d\tilde{\mu}(x)=\int_{Y_d} b(y)\cdot\nabla\ph_\sharp(y)\,d\mu(y),
\]
combined with Lemma~\ref{lem.phd}.
This concludes the proof of Proposition~\ref{pro.divcurl}.
\end{proof}
\begin{arem}\label{rem.divcurl}
Equation \eqref{dmub=0} can be considered as the divergence free of the vector-valued measure $\mu\,b$ in the torus $Y_d$, while equation \eqref{dtmub=0} is exactly the divergence free of the vector-valued measure $\tilde{\mu}\,b$ in the space $\R^d$. Equation \eqref{dmub=0} is also equivalent to
\beq\label{divcurl}
\forall\,(\nabla\psi)\in C^0_\sharp(Y_d)^d,\quad \int_{Y_d} b(y)\cdot\nabla\psi(y) \,d\mu(y)=\left(\int_{Y_d} b(y)\,d\mu(y)\right)\cdot
\left(\int_{Y_d} \nabla\psi(y)\,dy\right),
\eeq
since
\[
\nabla\psi\in C^0_\sharp(Y_d)^d\;\Leftrightarrow\; \left(x\mapsto \psi(x)-x\cdot{ \int_{Y_d} \nabla\psi(y)\,dy}\right)\in C^1_\sharp(Y_d).
\]
So, condition~\eqref{divcurl} may be regarded as a divergence-curl result involving the divergence free vector field $\mu\,b$ with the invariant probability measure $\mu$ and the gradient field $\nabla\psi$ with Lebesgue's measure.
\end{arem}
\subsection{The case of a nonlinear current field}\label{ss.1app}
As a direct consequence of Proposition~\ref{pro.Cb} and Proposition~\ref{pro.divcurl}, the following result gives the asymptotics of the ODE's flow \eqref{bX} when $b$ is a (non necessarily divergence free) nonlinear current field.
\begin{apro}\label{pro.FDv}
Let $v\in C^2(\R^d)$ be a function such that
\beq\label{Dv}
\nabla v\in C^1_\sharp(\R^d)^d\quad\mbox{and}\quad \#\left(\big\{x\in Y_d:\nabla v(x)=\overline{\nabla v}\big\}\right)<\infty,
\eeq
and let $F(x,\xi)\in C^1_\sharp(Y_d;C^1(\R^d))^d$ be a vector-valued function satisfying
\beq\label{FDv}
\left\{\ba{rl}
\forall\,x\in Y_d, & F(x,\overline{\nabla v})=0_{\R^d}
\\ \ecart
\forall\,(x,\xi,\eta)\in Y_d\times \R^d\times\R^d,\ \xi\neq \eta, & \big(F(x,\xi)-F(x,\eta)\big)\cdot(\xi-\eta)>0.
\ea\right.
\eeq
\par\noindent
Then, the vector field $b\in C^1_\sharp(Y_d)^d$ defined by
\beq\label{b=FDv}
b(x):=F(x,\nabla v(x))\quad\mbox{for }x\in Y_d,
\eeq
satisfies
\beq\label{cIFDv}
\sfC_b=\{0_{\R^d}\}.
\eeq
\end{apro}
\begin{proof}{ of Proposition~\ref{pro.FDv}}
Let $\mu$ be an invariant probability measure on $Y_d$ for the flow $X$ associated with the vector field $b$ \eqref{b=FDv}.
By virtue of the divergence-curl result \eqref{dmub=0} we have
\[
\int_{Y_d}b(x)\cdot\nabla v(x)\,d\mu(x)=\int_{Y_d}F(x,\nabla v(x))\cdot\nabla v(x)\,d\mu(x)=\left(\int_{Y_d}F(x,\nabla v(x))\,d\mu(x)\right)\cdot\overline{\nabla v}.
\]
This combined with $F(\cdot,\overline{\nabla v})=0_{\R^d}$ and the monotonicity~\eqref{FDv} yields
\beq\label{ineF}
\int_{Y_d}\underbrace{\big(F(x,\nabla v(x))-F(x,\overline{\nabla v})\big)\cdot\big(\nabla v(x)-\overline{\nabla v}\big)}_{\geq 0}\,d\mu(x)=0,
\eeq
which implies that
\[
\big(F(x,\nabla v(x))-F(x,\overline{\nabla v})\big)\cdot\big(\nabla v(x)-\overline{\nabla v}\big)=0\quad d\mu(x)\mbox{-a.e.}
\]
Hence, since $F(x,\cdot)$ is strictly monotonic in the sense of \eqref{FDv}, we deduce that
\[
\nabla v(x)=\overline{\nabla v}\quad d\mu(x)\mbox{-a.e.}
\]
Therefore, due to \eqref{Dv} the measure $\mu$ is a convex combination of the Dirac masses
\[
\mu=\sum_{x\in\{\nabla v=\overline{\nabla v}\}}c_x\,\delta_{x}\qquad\mbox{with}\qquad c_x\geq 0\quad\mbox{and}
\sum_{x\in\{\nabla v=\overline{\nabla v}\}}c_x=1.
\]
Finally, this combined with $F(\cdot,\overline{\nabla v})=0_{\R^d}$ implies that
\[
\int_{Y_d}b(y)\,d\mu(y)=\sum_{x\in\{\nabla v=\overline{\nabla v}\}}c_x\,F(x,\nabla v(x))
=\sum_{x\in\{\nabla v=\overline{\nabla v}\}}c_x\,F(x,\overline{\nabla v})=0_{\R^d},
\]
which leads us to \eqref{cIFDv}.
\end{proof}
\begin{Aexa}
A class of functions $F$ satisfying \eqref{FDv} is given by
\[
F(x,\xi)=\nabla_\xi f(x,\xi)\quad\mbox{for }(x,\xi)\in Y_d\times\R^d,
\]
where $f\in C^1_\sharp(Y_d;C^2(\R^d))$, and for any $x\in Y_d$, the function $f(x,\cdot)$ is strictly convex in $\R^d$ with $\overline{\nabla v}$ as unique minimizer.
\par
For example, the vector field $b$ is the linear current field
\[
b(x)=A(x)\nabla v(x)=\nabla_\xi f(x,\nabla v(x))\quad\mbox{for }x\in Y_d,
\]
when $f$ is the non negative quadratic functional defined by
\[
f(x,\xi):={1\over 2}\,A(x)\,\xi\cdot \xi\quad\mbox{for }(x,\xi)\in Y_d\times\R^d,
\]
for any non negative symmetric matrix-valued function $A\in C^1_\sharp(Y_d)^{d\times d}$.
Note that in this case, the finite set condition of \eqref{Dv} and the strict convexity of $f$ can be replaced by the unique condition $\overline{\nabla v}=0_{\R^d}$, or equivalently, $v$ is $\Z^d$-periodic.
\par
Indeed, let $\mu\in\cI_b$ be an invariant probability measure for the flow $X$.
Similarly as \eqref{ineF}, by the divergence-curl relation \eqref{dmub=0} we have
\[
\int_{Y_d}\underbrace{A(y)\nabla v(y)\cdot\nabla v(y)}_{\geq 0}\,d\mu(y)=\int_{Y_d}b(y)\cdot\nabla v(y)\,d\mu(y)
=\left(\int_{Y_d}b(y)\,d\mu(y)\right)\cdot\overline{\nabla v}=0,
\]
which implies that $A\nabla v\cdot\nabla v=0$ $\mu$-a.e. in $Y_d$.
However, since the matrix-valued $A$ is symmetric and non negative, from the Cauchy-Schwarz inequality we deduce that $A\nabla v=0$ $\mu$-a.e. in $Y_d$, and thus
\[
\int_{Y_d}b(y)\,d\mu(y)=\int_{Y_d}A(y)\nabla v(y)\,d\mu(y)=0_{\R^d}.
\]
Therefore, we obtain the desired equality \eqref{cIFDv}.
\end{Aexa}
%%%%%%%%%%
\section{Some new results involving the singleton condition}\label{s.newres}
\subsection{A perturbation result}\label{ss.mres}
The main result of the paper is the following.
\begin{atheo}\label{thm.Cbsin}
Let $b\in C^1_\sharp(Y_d)^d$ be such that $b = \rho\,\Phi$, where $\rho$ is a non negative not null function in $C^1_\sharp(Y_d)$ and $\Phi$ is a non vanishing vector field in $C^1_\sharp(Y_d)^d$. Assume that there exists a sequence $(\rho_n)_{n\in\N}\in C^1_\sharp(Y_d)^\N$ such that 
\begin{enumerate}[$(i)$]
\item for any $n\in\N$, $\rho\leq \rho_n >0$ in $Y_d$, and $(\rho_n)_{n\in\N}$ converges uniformly to $\rho$ on $Y_d$,
\item for any $n\in\N$, $\sfC_{b_n} = \{\zeta_n\}$ for some $\zeta_n\in\R^d$, where $b_n:=\rho_n\,\Phi$.
\end{enumerate}
Then, the sequence $(\zeta_n)_{n\in\N}$ converges to some $\zeta\in\R^d$.
Moreover, we have the following alternative:
\begin{itemize}
\item If $\rho$ is positive in $Y_d$,  then $\sfC_{b}=\{\zeta\}$ with $\zeta\neq 0_{\R^d}$.
\item If $\rho$ vanishes in $Y_d$, then $\sfC_b=[0_{\R^d},\zeta]$. Moreover, we have $\{0_{\R^d},\zeta\}\subset\sfA_b$.
\end{itemize}
\end{atheo}
\begin{arem}
In dimension two and in the second alternative of Theorem~\ref{thm.Cbsin}, it is not surprising to obtain for the rotation set $\sfC_b$ a closed line segment with one end at $0_{\R^2}$.
Indeed, Franks and Misiurewicz \cite[Theorem~1.2]{FrMi} proved that the rotation set of any two-dimensional continuous flow is always a closed line segment of a line passing through $0_{\R^2}$. Moreover, this segment has one end at $0_{\R^2}$, when it also has an irrational slope.
However, our perturbation result provides such a closed line segment in any dimension and for any non zero slope.
\end{arem}
\begin{proof}{ of Theorem~\ref{thm.Cbsin}}
First of all, assume that $\sfC_b\neq\{0_{\R^d}\}$, namely there exists a probability measure $\mu\in\cI_b$ satisfying
\[
\int_{Y_d}b(x)\,d\mu(x)\neq 0.
\]
This implies that
\beq\label{intrhomu}
\int_{Y_d} \rho(x)\,d\mu(x) = \int_{Y_d}{|b(x)|\over|\Phi(x)|}\,d\mu(x) >0.
\eeq
For every $n\in\N$, define the probability measure $\mu_n$ on $Y_d$ by
\beq\label{munrhon}
d\mu_n(x) := C_n\, {\rho(x)\over \rho_n(x)}\,d\mu(x),\quad \text{where} \quad C_n := \left(\int_{Y_d}{\rho(y)\over \rho_n(y)}\,d\mu(y)\right)^{-1}\in(0,\infty)
\eeq
due to $\rho_n > 0$ and \eqref{intrhomu}.
Since $\mu\in\cI_b$, it follows from equality \eqref{dmub=0} that 
\[
\forall\,\ph\in C^1_\sharp(Y_d),\quad \int_{Y_d} \rho_n(x)\,\Phi(x)\cdot\nabla\ph(x)\,d\mu_n(x)
= C_n\int_{Y_d} \rho(x)\,\Phi(x)\cdot\nabla\ph(x)\,d\mu(x)=0,
\]
hence by the equivalence $(i)$-$(iii)$ of Proposition~\ref{pro.divcurl}, $\mu_n\in\cI_{b_n}$. Then, from assumption $\sfC_{b_n} = \{\zeta_n\}$ we deduce that
\[
\zeta_n = \int_{Y_d}b_n(x)\, d\mu_n(x) = C_n \int_{Y_d}\rho_n(x)\, \Phi(x){\rho(x)\over \rho_n(x)}\,d\mu(x) = C_n \int_{Y_d}b(x)\, d\mu(x).
\]
Moreover, by Lebesgue's theorem and \eqref{intrhomu} we have
\beq\label{cmu}
\lim_{n\to\infty}C_n = c_\mu:=\left(\int_{\{\rho>0\}}d\mu(x)\right)^{-1}\in[1,\infty).
\eeq
Therefore, the sequence $(\zeta_n)_{n\in\N}$ converges to some $\zeta\in\R^d$ which is independent of $\mu$ and satisfies the equality
\beq\label{zetacmu}
\zeta=c_\mu \int_{Y_d}b(x)\, d\mu(x),\quad\mbox{for any $\mu\in\cI_b$ with }\int_{Y_d}b(x)\, d\mu(x)\neq 0.
\eeq
\par\noindent
$\bullet$ If $\rho$ is positive in $Y_d$, then \eqref{intrhomu} and $c_\mu=1$ hold for any $\mu\in\cI_b$. This combined with \eqref{zetacmu} implies that $\sfC_{b}\setminus\{0_{\R^d}\}=\{\zeta\}$. Therefore, due to the convexity of $\sfC_b$ we obtain that $\sfC_{b}=\{\zeta\}$.
\par\noindent
$\bullet$ On the contrary, assume that there exists $\al\in Y_d$ such that $\rho(\alpha)=0$. Since the Dirac distribution $\delta_\alpha$ at $\alpha$ is invariant for the flow associated with $b$, we have
\[
0_{\R^d}=\int_{Y_d}b(x)\,d\delta_{\alpha}(x)\in \sfC_b.
\]
If the sequence $(\zeta_n)_{n\in\N}$ converges to~$0_{\R^d}$, then we have $\sfC_b=\{0_{\R^d}\}$.
Indeed, if $\sfC_b\neq \{0_{\R^d}\}$, then the first step of the proof implies \eqref{cmu} and \eqref{zetacmu} for some $\mu\in\cI_b$ with $\zeta=\lim_n \zeta_n\neq 0_{\R^d}$, which yields a contradiction.
\par
Now, assume that the sequence $(\zeta_{n})_{n\in\N}$ does not converge to~$0_{\R^d}$, and consider a subsequence $(\zeta_{n_k})_{k\in\N}$ which converges to some $\hat{\zeta}\neq 0_{\R^d}$.
Up to extract a new subsequence, we may assume that the sequence $(\mu_{n_k})_{k\in\N}$ converges weakly~$*$ to some probability measure $\hat{\mu}$ on $Y_d$.
Passing to the limit in the divergence-curl relation \eqref{dmub=0} satisfied by $\mu_{n_k}\in\cI_{b_{n_k}}$:
\beq\label{bkmun}
\forall\,\ph\in C^1_\sharp(Y_d),\quad \int_{Y_d} b_{n_k}(x)\cdot\nabla\ph(x)\,d\mu_{n_k}(x)=0,
\eeq
we get that
\beq\label{bmuh}
\forall\,\ph\in C^1_\sharp(Y_d),\quad \int_{Y_d} b(x)\cdot\nabla\ph(x)\,d\hat{\mu}(x)=0,
\eeq
so that $\hat{\mu}\in\cI_b$ again by the equivalence $(i)$-$(iii)$ of Proposition~\ref{pro.divcurl}. Moreover, by the uniform convergence of the sequence $(b_n)_{n\in\N}$ to~$b$ we have
\beq\label{zetakn}
0_{\R^d}\neq \hat{\zeta}=\lim_{k\to\infty}\zeta_{n_k}=\lim_{k\to\infty}\int_{Y_d} b_{n_k}(x)\,d\mu_{n_k}(x)=\int_{Y_d} b(x)\,d\hat{\mu}(x)\in \sfC_b\setminus\{0_{\R^d}\}.
\eeq
Therefore, we may apply the first step of the proof with $\mu=\hat{\mu}$.
Hence, the whole sequence $(\zeta_{n})_{n\in\N}$ converges to $\zeta=\hat{\zeta}\in \sfC_b\setminus\{0_{\R^d}\}$.
Moreover, by \eqref{cmu} equality \eqref{zetacmu} also reads as
\beq\label{zecmu}
\int_{Y_d}b(x)\, d\mu(x)={1\over c_\mu}\,\zeta \in [0_{\R^d},\zeta],\quad\mbox{for any $\mu\in\cI_b$ with }\int_{Y_d}b(x)\, d\mu(x)\neq 0_{\R^d}.
\eeq
This combined with $0_{\R^d}\in \sfC_b$, $\zeta\in \sfC_b\setminus\{0_{\R^d}\}$ and the convexity of $\sfC_b$ implies that $\sfC_b=[0_{\R^d},\zeta]$.
\par\noindent
Now, let us prove that $\{0_{\R^d},\zeta\}\subset\sfA_b$. On the one hand, since since $X(\cdot,\alpha)=\alpha$, we have
\[
\lim_{t\r\infty} {X(t,\alpha)\over t}=0_{\R^d},
\]
hence $0_{\R^d}\in \sfA_b$.
On the other hand, by equality \eqref{lim-mu-pp-zeta} in Lemma~\ref{lem.Cbnsin} below (based on the above proved formulas \eqref{cmu} and \eqref{zecmu}), we get immediately that $\zeta\in\sfA_b$.
\par
Finally, it remains to deal with the case $\sfC_b=\{0_{\R^d}\}$. We have just to prove that the sequence $(\zeta_n)_{n\in\N}$ converges to $0$. 
Repeating the argument between \eqref{bkmun} and \eqref{zetakn} for any subsequence $(\zeta_{n_k})_{k\in\N}$ which converges to $\hat{\zeta}\in\R^d$ and any subsequence $(\mu_{n_k})_{k\in\N}$ which converges weakly~$*$ to $\hat{\mu}$ in $\cM_p(Y_d)$, we get that $\hat{\mu}\in\cI_b$ and
\[
\hat{\zeta}=\lim_{k\to\infty}\zeta_{n_k}=\lim_{k\to\infty}\int_{Y_d} b_{n_k}(x)\,d\mu_{n_k}(x)=\int_{Y_d} b(x)\,d\hat{\mu}(x)=0_{\R^d}.
\]
Therefore, the whole sequence $(\zeta_n)_{n\in\N}$ converges to $0_{\R^d}$.
The proof of Theorem~\ref{thm.Cbsin} is now complete.
\end{proof}
\begin{arem}\label{rem.Cbnsin}
In the setting of Theorem~\ref{thm.Cbsin}, Example~\ref{exa.2} below provides a two-dimensional case of vector field $b=\rho\,\Phi$ with a vanishing function $\rho$, in which $\sfC_b$ is a closed line segment of~$\R^2$ not reduced to a singleton, and thus $\#\sfA_b\geq 2$.
\end{arem}
\begin{alem}\label{lem.Cbnsin}
Let $b$ be a vector field satisfying the assumptions of Theorem~\ref{thm.Cbsin} with a vanishing function $\rho$.
Then,  for any $\mu\in\cI_b$ with $\int_{Y_d} b(x)\, d\mu(x) \neq 0$, we have 
\begin{equation}\label{lim-mu-pp-zeta}
\lim_{t\r\infty} \frac{X(t,x)}{t} = \zeta \quad \text{for  $\mu$-a.e. }x\in \{\rho > 0\}.
\end{equation}
\end{alem}
\begin{proof}{ of Lemma~\ref{lem.Cbnsin}}
Let $\mu\in\cI_b$ with $\int_{Y_d} b(x)\, d\mu(x) \neq 0$. By Birkhoff's theorem, there exists a measurable function $x\mapsto g(x)$ 
such that by \eqref{Ab} and \eqref{ACb}
\[
\lim_{t\r\infty} \frac{X(t,x)}{t} = g(x)\in \sfA_b\subset \sfC_b \quad \text{for $\mu$-a.e. }x\in Y_d.
\]
(Also see Proposition A.1 in the Appendix to get that $g(x)\in\sfC_b$). 
Hence, by virtue of Theorem~\ref{thm.Cbsin} we have
\begin{equation} \label{lim-d-zeta}
\lim_{t\r\infty} \frac{X(t,x)}{t} = \Delta(x)\, \zeta \quad \text{for $\mu$-a.e. }x\in Y_d,
\end{equation}
where $\Delta :Y_d\to [0,1]$ is a measurable function such that $\Delta(x)=0$ if $\rho(x)=0$.
Applying successively convergence \eqref{lim-d-zeta}, Lebesgue's theorem, Fubini's theorem and the invariance of the measure $\mu$, we get that 
\[
\ba{l}
\dis \left(\int_{Y_d} \Delta(x)\, d\mu(x)\right)\zeta = \lim_{t\r\infty}\int_{Y_d}\frac{X(t,x)}{t}\,d\mu(x) = \lim_{t\r\infty}\int_{Y_d}\left({1\over t}\int_0^t b(X(s,x))\,ds\right)d\mu(x)
\\ \ecart
\dis = \lim_{t\r\infty}{1\over t}\int_0^t\left(\int_{Y_d} b(X(s,x))\,d\mu(x)\right)ds = \int_{Y_d} b(x)\, d\mu(x),
\ea
\]
or equivalently,
\[
\int_{\{\rho > 0\}} b(x)\, d\mu(x) = \left(\int_{\{\rho > 0\}} \Delta(x)\, d\mu(x)\right)\zeta.
\]
Moreover, since $\int_{Y_d} b(x)\, d\mu(x) \neq 0$, from the formulas \eqref{cmu} and \eqref{zecmu} in the proof of Theorem~\ref{thm.Cbsin} we deduce that 
\[
\int_{Y_d} b(x)\, d\mu(x) = \int_{\{\rho > 0\}} b(x)\, d\mu(x) = \beta_\mu\,\zeta,\quad \text{where}\quad \beta_\mu := \int_{\{\rho>0\}}d\mu(x),
\]
so that
\[
\int_{\{\rho>0\}}d\mu(x) = \int_{\{\rho > 0\}} \Delta(x)\, d\mu(x),
\]
or equivalently,
\[
\int_{\{\rho>0\}}\big(1-\Delta(x))\, d\mu(x) = 0.
\]
Since the function $\Delta$ takes values in $[0,1]$, the former equality implies that $\Delta=1$ for $\mu$-a.e. in $\{\rho > 0\}$,
which combined with \eqref{lim-d-zeta} yields the desired limit \eqref{lim-mu-pp-zeta}.
\end{proof}
\begin{arem}\label{rem.Phindep}
Theorem~\ref{thm.Cbsin} cannot be extended to the more general case where the direction $\Phi$ of $b_n$ also depends on $n$, as shown in Example~\ref{exa.1} below. More precisely, the independence of $\Phi$ with respect to $n$ is crucial for the proof of Theorem~\ref{thm.Cbsin} to build in \eqref{munrhon} the invariant probability measure $\mu_n\in\cI_{b_n}$ from a given invariant probability measure $\mu\in\cI_b$.
\end{arem}
\begin{acor}\label{cor.AbCb}
Let $\Phi$ be a non vanishing vector field in $C^1_\sharp(Y_d)^d$.
Then, we have for the uniform convergence topology in $C^0_\sharp(Y_d)$,
\beq\label{Cbrho>0}
\ba{c}
\dis \big\{\rho\in C^1_\sharp(Y_d):\rho>0\mbox{ and }\#\,\sfC_{\rho\,\Phi}=1\big\}=\big\{\rho\in C^1_\sharp(Y_d):\rho>0\mbox{ and }\#\,\sfA_{\rho\,\Phi}=1\big\}
\\ \ecart
\mbox{ is a closed subset of }\big\{\rho\in C^1_\sharp(Y_d):\rho>0\big\}.
\ea
\eeq
\end{acor}
\begin{proof}{ of Corollary~\ref{cor.AbCb}}
The equality of the sets in \eqref{Cbrho>0} follows directly from \eqref{ACb}.
Moreover, the fact that the first set is closed is a straightforward consequence of the first case of Theorem~\ref{thm.Cbsin}.
\end{proof}
\begin{arem}\label{rem.AbCb}
Example~1 of \cite{LlMa} provides a two-dimensional case where the field $b^*=\Phi^*$, {\em i.e.} $\rho^*=1$, is such that $\#\sfA_{\Phi^*}=2$. Therefore, by \eqref{Cbrho>0} there exists an open ball $B^*$ centered at $\rho^*=1$ such that $\#\sfA_{\rho\,\Phi^*}>1$ for any $\rho\in B^*$.
\par
On the contrary, the assertion \eqref{Cbrho>0} of Corollary~\ref{cor.AbCb} does not hold in general when condition $\rho>0$ is enlarged to condition $\rho\geq 0$.
Indeed, in the setting of Theorem~\ref{thm.Cbsin} the two-dimensional Example~\ref{exa.2} provides a sequence of fields $(b_n=\rho_n\,\Phi)_{n\geq 1}$ which converges uniformly in $Y_2$ to some field $b=\rho\,\Phi$, so that $\sfC_b$ is not reduced to a singleton while $C_{b_n}$ is a singleton for any $n\geq 1$. Therefore, in this case the set
\[
\dis \big\{r\in C^1_\sharp(Y_d):r\geq 0\mbox{ and }\#\,\sfC_{r\Phi}=1\big\}=\big\{r\in C^1_\sharp(Y_d):r\geq 0\mbox{ and }\#\,\sfA_{r\phi}=1\big\}
\]
is not closed in $\big\{r\in C^1_\sharp(Y_d):r\geq 0\big\}$.
\end{arem}
\subsection{Applications to the asymptotics of the ODE's flow}\label{ss.appflow}
The first result provides the asymptotics of the flow $X(\cdot,x)$ at any point $x$ whose orbit does not meet the set $\{\rho=0\}$ in $Y_d$.
\begin{acor}\label{cor.dist-orbite}
Assume that the conditions of Theorem~\ref{thm.Cbsin} hold with a vanishing function~$\rho$.
Denote by $\pi$ the canonical projection from $\R^d$ on the torus $Y_d$, and define for $x\in\R^d$, 
\beq\label{Fx}
F_x := \overline{\pi\big(X(\R,x)\big)}^{Y_d},
\eeq
{\em i.e.} the closure in $Y_d$ of the projection on $Y_d$ of the orbit $X(\R,x)$ of $x$.
Then, we have
\begin{equation}\label{lim-zeta-support}
\forall\,x\in Y_d,\quad F_x \cap \{\rho=0\} = \mbox{{\rm \O}}\;\Rightarrow\; \lim_{t\r\infty} \frac{X(t,x)}{t} = \zeta.
\end{equation}
\end{acor}
\begin{proof}{ of Corollary~\ref{cor.dist-orbite}}
Let $x\in Y_d$ be such that $F_x \cap \{\rho=0\} = \mbox{{\rm \O}}$.
First, by virtue of Proposition~\ref{pro.invmeas} with $g=b_i$ and $x_n=x$, any limit point derived from the asymptotics $X(t,x)/t$ as $t\to\infty$, is of the form
\[
\int_{Y_d} b(y)\, d\mu(y)\quad\mbox{for some }\mu \in\cI_b.
\]
Let us prove that the support of such an invariant probability measure $\mu$ is contained in the closed set $F_x$.
By Lemma~\ref{lem-extrac} $\mu$ is a limit point for the weak-$*$ of some sequence $(\nu_n)_{\in\N}$  in $\cM_p(Y_d)$ given by 
\[
\int_{Y_d} f(y)\, d\nu_n(y) = \frac{1}{r_n} \int_0^{r_n} f(X(s,x))\,ds\quad\mbox{for }f\in C^0_\sharp(Y_d).
\]
Let $f\in C^0_\sharp(Y_d)$ which is zero in $F_x$.
Then, since $\pi(X(s,x)) \in F_x$ for any $s\in\R$, we have
\[
\int_{Y_d} f(y)\, d\nu_n(y) = 0.
\]
Passing to the limit in the previous equality we get that for any $f\in C^0_\sharp(Y_d)$ which is zero in $F_x$,
\[
\int_{Y_d} f(y)\, d\mu(y) = 0,
\]
which means  that the support of $\mu$ is contained in $F_x$.
\par
Now, let $a$ be a limit point of $X(t,x)/t$ as $t\to\infty$.
Then, we have
\[
a=\int_{Y_d} b(y)\, d\mu(y)\quad\mbox{with}\quad\mbox{Supp}(\mu)\subset F_x\subset \{\rho>0\}.
\]
Hence, the constant $c_\mu$ of \eqref{cmu} is equal to $1$, so that by \eqref{zetacmu} we get that
\[
\int_{Y_d} b(y)\, d\mu(y) = \zeta.
\]
Therefore, any limit point $a$ is equal to the constant $\zeta$ obtained in Theorem~\ref{thm.Cbsin}, which implies that
\[
\lim_{t\r\infty} \frac{X(t,x)}{t} = \zeta.
\]
The desired implication \eqref{lim-zeta-support} is thus established.
\end{proof}
\par\bigskip
The following result provides some condition on the direction $\Phi$ of the vector field $b=\rho\,\Phi$, which allows us to specify the result of Theorem~\ref{thm.Cbsin} when the function $\rho$ vanishes.
\begin{acor}\label{cor.siPhidiv0}
Consider a vector field $b=\rho\,\Phi\in C^1_\sharp(Y_d)^d$, a sequence $(\rho_n)_{n\in\N}\in C^1_\sharp(Y_d)^\N$ and $b_n=\rho_n\,\Phi$ satisfying the assumptions of Theorem~\ref{thm.Cbsin}. Also assume that there exists a positive function $\si\in C^1_\sharp(Y_d)$ such that $\sigma\,\Phi$ is divergence free in $\R^d$, and that $\rho$ vanishes in $Y_d$.
\par\noindent
We have the following alternative involving the harmonic mean $\underline{\rho/\sigma}$ of $\rho/\sigma$:
\begin{itemize}
\item If $\underline{\rho/\sigma}=0$, then the flow $X$ associated with $b$ satisfies the asymptotics
\beq\label{asy0}
\forall\,x\in Y_d,\quad \lim_{t\r\infty} {X(t,x)\over t}=0_{\R^d}.
\eeq
\item If $\underline{\rho/\sigma}>0$, then we have
\beq\label{Cbrhosi}
\sfC_b=[0_{\R^d},\zeta]\quad\mbox{with}\quad \zeta:=\underline{\rho/\sigma}\int_{Y_d}\Phi(y)\,\sigma(y)\,dy.
\eeq
\end{itemize}
\end{acor}
\begin{proof}{ of Corollary~\ref{cor.siPhidiv0}}
Since $\sigma\,\Phi$ is divergence free in $\R^d$ and $\sigma$ is $\Z^d$-periodic, by virtue of Proposition~\ref{pro.divcurl} the probability measure on $Y_d$: $\sigma(x)/\overline{\sigma}\,dx$, where $\overline{\sigma}>0$ is the arithmetic mean of $\sigma$, is an invariant probability measure for the flow associated with the vector field $\Phi$, {\em i.e.} $\sigma(x)/\overline{\sigma}\,dx\in\cI_\Phi$.
\par\noindent
For every $n\in\N$, define the probability measure $\mu_n$ on $Y_d$ by
\beq\label{munCn}
d\mu_n(x):={C_n\over\rho_n(x)}\,\sigma(x)\,dx\quad\mbox{where}\quad C_n:=\left(\int_{Y_d}{1\over\rho_n(y)}\,\sigma(y)\,dy\right)^{-1}.
\eeq
Due to $\sigma(x)/\overline{\sigma}\,dx\in\cI_\Phi$ we have by Proposition~\ref{pro.divcurl}
\[
\forall\,\psi\in C^1_\sharp(Y_d),\quad \int_{Y_d}b_n(x)\cdot\nabla\psi(x)\,d\mu_n(x)=C_n\int_{Y_d}\Phi(x)\cdot\nabla\psi(x)\,\sigma(x)\,dx=0,
\]
which again by Proposition~\ref{pro.divcurl} implies that $\mu_n\in\cI_{b_n}$. This combined with the singleton assumption $\sfC_{b_n}=\{\zeta_n\}$ yields
\beq\label{zetanPhi}
\zeta_n=\int_{Y_d}b_n(x)\,d\mu_n(x)=C_n\int_{Y_d}\,\Phi(x)\,\sigma(x)\,dx.
\eeq
\begin{itemize}
\item If $\underline{\rho/\sigma}=0$, then by Fatou's lemma we get that
\[
\infty=\int_{Y_d}{\sigma(y)\over\rho(y)}\,dy\leq\liminf_{n\to\infty}\left(\int_{Y_d}{\sigma(y)\over\rho_n(y)}\,dy\right),
\]
which implies that the sequence $(C_n)_{n\in\N}$ tends to $0$. Hence, by \eqref{zetanPhi} the sequence $(\zeta_n)_{n\in\N}$ converges to $\zeta=0_{\R^d}$. Therefore, by the second case of Theorem~\ref{thm.Cbsin} we have $\sfC_b=\{0_{\R^d}\}$, or equivalently by \eqref{Cbsin}, the null asymptotics~\eqref{asy0} holds.
\item If $\underline{\rho/\sigma}>0$, then from the convergence of $\rho_n$ to $\rho$, the inequality $\rho_n\geq \rho$ and Lebesgue's theorem we deduce that
\[
\lim_{n\to\infty}\,\int_{Y_d}{\sigma(y)\over\rho_n(y)}\,dy=\int_{Y_d}{\sigma(y)\over\rho(y)}\,dy={1\over \underline{\rho/\sigma}}<\infty,
\]
which by \eqref{munCn}, \eqref{zetanPhi} implies that
\[
\zeta=\lim_{n\to\infty}\zeta_n=\underline{\rho/\sigma}\int_{Y_d}\Phi(y)\,\sigma(y)\,dy.
\]
Therefore, again by the second case of Theorem~\ref{thm.Cbsin} we obtain the set $\sfC_b$ \eqref{Cbrhosi}.
\end{itemize}
\end{proof}
\begin{arem}\label{rem.hrho}
In the two-dimensional case of Corollary~\ref{cor.siPhidiv0}, if $\rho$ is in $C^2_\sharp(Y_2)$ and vanishes at some point $x_0\in Y_2$, then the harmonic mean $\underline{\rho/\sigma}$ is $0$.
Indeed, since $\rho$ is non negative, $x_0$ is a critical point of $\rho$.
Hence, we get that for any $x$ close to $x_0$,
\[
\rho(x)={1\over 2}\,\nabla^2\rho(x_0)\,(x-x_0)\cdot(x-x_0)+o(|x-x_0|^2)\quad\mbox{and thus}\quad {\sigma(x)\over\rho(x)}\geq {C\over |x-x_0|^2}.
\]
which implies that $\sigma/\rho\notin L^1(Y_2)$ and $\underline{\rho/\sigma}=0$.
Therefore, the null asymptotics \eqref{asy0} holds. 
\par
Otherwise, if
\[
\#\big\{x\in Y_2:\rho(x)=0\big\}\in(0,\infty),
\]
and if for any $x_0\in Y_2$ with $\rho(x_0)=0$ we have for any $x$ close to $x_0$,
\[
\rho(x)\geq c_0\,|x-x_0|^{\alpha_0},\quad\mbox{for some }\alpha_0\in(1/2,1)\mbox{ and }c_0>0,
\]
(note that the former condition remains compatible with $\rho\in C^1_\sharp(Y_2)$), then $\underline{\rho/\sigma}>0$.
Indeed, we are led by a translation to $x_0=0_{\R^2}$, and passing to polar coordinates we deduce that any $r_0\in(0,1/2)$,
\[
\int_{\{x\in Y_2:|x|<r_0\}}{dx\over |x|^{\alpha_0}}=2\pi\int_0^{r_0}{dr\over r^{2\alpha_0-1}}<\infty.
\]
Therefore, we get the full closed line segment \eqref{Cbrhosi} for the limit set $\sfC_b$.
\par
In Example~\ref{exa.2} below we will provide a two-dimensional example of such an enlarged limit set $\sfC_b$ obtained from a sequence of singleton sets $(\sfC_{b_n})_{n\in\N}$ where $b_n=\rho_n\,\Phi$ has a fixed direction~$\Phi$.
These results also apply to Corollary~\ref{cor.CbP} replacing $1/\sigma$ by $a$.
\end{arem}
\par
The next result uses both Theorem~\ref{thm.Cbsin}, Corollary~\ref{cor.siPhidiv0} and the two-dimensional ergodic approach of~\cite{Pei}. 
\begin{acor}\label{cor.CbP}
Let $b$ be a two-dimensional vector field in $C^1_\sharp(Y_2)^2$ such that $b = \rho\,\Phi$, with $\rho$ a non zero  non negative function in $C^1_\sharp(Y_2)$, and
\beq\label{PhiaDu}
\Phi=a\,R_\perp\nabla u\;\;\mbox{in }Y_2,
\eeq
where $a$ is a positive function in $C^1_\sharp(Y_2)$, $\nabla u$ is a non vanishing gradient field in $C^1_\sharp(Y_2)^2$ and $R_\perp$ is the $-\,\pi/2$ rotation matrix of $\R^{2\times 2}$. Also assume that
\beq\label{ergDu}
\forall\,\kappa\in\Z^2\setminus\{0_{\R^2}\},\;\;\left(\int_{Y_2}\nabla u(y)\,dy\right)\cdot \kappa\neq 0.
\eeq
We have the following cases:
\begin{itemize}
\item If $\rho$ is positive in $Y_2$, then the flow $X$ associated with $b$ satisfies the asymptotics
\beq\label{asybsi}
\forall\,x\in Y_2,\quad \lim_{t\r\infty} \frac{X(t,x)}{t} = \underline{a\rho}\int_{Y_2}R_\perp\nabla u(y)\,dy.
\eeq
\item If $\rho$ vanishes in $Y_2$ and $\underline{a\rho}=0$, then the flow $X$ satisfies the null asymptotics
\beq\label{asybsi0}
\forall\,x\in Y_2,\quad \lim_{t\r\infty} \frac{X(t,x)}{t} = 0_{\R^2}.
\eeq
\item  If $\rho$ vanishes in $Y_2$ and $\underline{a\rho}>0$, then the set $\sfC_b$ is given by
\beq\label{CbPhi}
\sfC_b=[0_{\R^2},\zeta]\quad\mbox{with}\quad \zeta:=\underline{a\rho}\int_{Y_2}R_\perp\nabla u(y)\,dy\neq 0_{\R^2}.
\eeq
\end{itemize}
\end{acor}
\begin{proof}{ of Corollary~\ref{cor.CbP}} Consider the sequence $(\sigma_n)_{n\in\N}$ defined by
\beq\label{sin}
\sigma_n:={C_n\over a\rho_n}\quad\mbox{where}\quad C_n=\underline{a\rho_n}:=\left(\int_{Y_2}{dy\over (a\rho_n)(y)}\right)^{-1},
\eeq
where $(\rho_n)_{n\in\N}$ is a sequence in $C^1_\sharp(Y_2)^\N$ satisfying the condition $(i)$ of Theorem~\ref{thm.Cbsin}.
\par
First, note that by definition \eqref{sin} the vector field $b_n:=\rho_n\,\Phi$ satisfies
\beq\label{sibcnDu}
\sigma_n\,b_n = {C_n\over a}\,\Phi = C_n\, R_\perp\nabla u
\eeq
(recall that $C_n$ is a positive constant) so that $b_n$ is orthogonal to $\nabla u$.
Hence, the function $u$ is invariant by the flow $X_n$ associated with $b_n$, {\em i.e.}
\beq\label{invuXn}
\forall\,x\in Y_2,\ \forall\,t\in\R,\quad u(X_n(t,x))=u(x).
\eeq
This combined with the irrationality (or ergodicity) condition \eqref{ergDu} implies that the flow $X_n$ has no periodic solution in $Y_2$ according to \eqref{XperYd}. Otherwise, there exists $x\in Y_2$, $T>0$, and $\kappa\in\Z^2$ such that $X_n(T,x)=x+\kappa$, hence it follows that
\beq\label{invuXnT}
u(x)=u(X_n(T,x))=u(x+\kappa).
\eeq
Moreover, since $\nabla u$ is $\Z^2$-periodic, the function
\[
z\;\longmapsto \;u(z)-\left(\int_{Y_2}\nabla u(y)\,dy\right)\cdot z\quad\mbox{is $\Z^2$-periodic}.
\]
This combined with \eqref{invuXnT} yields
\[
\left(\int_{Y_2}\nabla u(y)\,dy\right)\cdot \kappa=0.
\]
Hence, from the incommensurability condition of \eqref{ergDu} we deduce that $\kappa=0_{\R^2}$. The flow $X_n(\cdot, x)$ is thus $T$-periodic, namely $X_n(\R,x)$ is a closed orbit. However, by virtue of the preliminary remark of the proof of \cite[Theorem~3.1]{Pei} this leads us to a contradiction since $b_n=\rho_n\,\Phi$ is non vanishing. Therefore, the flow $X_n$ associated with $b_n$ has no periodic solution in $Y_2$. Then, using the second step of the proof of \cite[Theorem~3.1]{Pei} we get the existence of a vector $\zeta_n\in\R^2$ such that 
\beq\label{Xnzen}
\forall\,x\in Y_2,\quad \lim_{t\r\infty} \frac{X_n(t,x)}{t} = \zeta_n,
\eeq
or equivalently by \eqref{Cbsin}, $\sfC_{b_n}=\{\zeta_n\}$, namely the condition $(ii)$ of Theorem~\ref{thm.Cbsin} holds.
Moreover, due to \eqref{sibcnDu} $\sigma_n\,b_n$ is divergence free in $\R^2$, or equivalently, in the torus sense \eqref{dmub=0}
\[
\forall\,\psi\in C^1_\sharp(Y_d),\quad \int_{Y_d} b_n(x)\cdot\nabla\psi(x)\,\sigma_n(x)\,dx=0.
\]
Hence, by virtue of Proposition~\ref{pro.divcurl} $\sigma_n(x)\,dx$ is an invariant probability measure for the flow $X_n$, which combined with \eqref{sibcnDu} implies that
\beq\label{sibzen}
\int_{Y_2}\sigma_n(y)\,b_n(y)\,dy=C_n\int_{Y_2}R_\perp\nabla u(y)\,dy\in \sfC_{b_n}=\{\zeta_n\}.
\eeq
Let us conclude:
\begin{itemize}
\item If $\rho$ is positive in $Y_2$, then from equality \eqref{sibzen} and the uniform convergence of $\rho_n$ to $\rho>0$ we deduce that
\[
\lim_{n\to\infty}\zeta_n=\zeta:=C\int_{Y_2}R_\perp\nabla u(y)\,dy,\quad\mbox{where}\quad C:=\lim_{n\to\infty}C_n=\underline{a\rho}.
\]
Therefore, by the first result of Theorem~\ref{thm.Cbsin} we get that $\sfC_b=\{\zeta\}$, or equivalently by~\eqref{Cbsin}, asymptotics \eqref{asybsi} is satisfied.
\item Otherwise, $\rho$ vanishes in $Y_2$. Moreover, by \eqref{PhiaDu} the vector field $a^{-1}\Phi$ is clearly divergence free in~$\R^2$.
Therefore, by virtue of Corollary~\ref{cor.siPhidiv0} with $\sigma=a^{-1}$, we deduce the null asymptotics \eqref{asybsi0} if $\underline{a\rho}=0$,
and the set $\sfC_b$ \eqref{CbPhi} if $\underline{a\rho}>0$.
The fact that $\zeta\neq 0_{\R^2}$ in \eqref{CbPhi} follows immediately from the ergodic condition \eqref{ergDu}.
\end{itemize}
The proof is now complete.
\end{proof}
\begin{arem}\label{rem.SP}
The condition \eqref{PhiaDu} on the direction $\Phi$ of the field $b = \rho\,\Phi$ may seem to be quite restrictive at the first glance.
Actually, we can deduce \eqref{PhiaDu} from the existence of a function $v\in C^2(Y_2)$ and a constant $c>0$ satisfying the inequality
\beq\label{PhiDvc}
\Phi\cdot\nabla v\geq c\;\;\mbox{in }\R^2.
\eeq
This inequality means that the equipotential $\{v=0\}$ (or any equipotential of $v$) is transverse to each orbit $Y(\R,x)$, $x\in \R^2$, of the flow $Y$ associated with $\Phi$. In other words, the equipotential $\{v=0\}$ can be regarded as a Siegel's curve \cite[Lemma~3]{Sie} for the flow $Y$ in $\R^2$ rather than in the torus~$Y_2$. Assuming that $b$ is non vanishing and that the flow $X$ associated with $b$ has no periodic trajectory in $Y_2$ according to~\eqref{XperYd} (which is an ergodic type condition) and using a Siegel's curve in the torus, Peirone \cite[Theorem~3.1]{Pei} has proved that the asymptotics of the flow $X(\cdot,x)$ does exist for any $x\in Y_2$ and is independent of $x$, or equivalently by \eqref{Cbsin}, that $\sfC_b$ is a singleton. Therefore, condition \eqref{PhiaDu} and the ergodic condition \eqref{ergDu} play the same role for the vector field $\Phi$ than Peirone's conditions for the vector field $b$ through a similar Siegel's curve approach. However, working with the non vanishing vector field $\Phi$ rather than $b$ allows us to obtain some new asymptotics when the vector field $b = \rho\,\Phi$ does vanish with the non negative function $\rho$.
\par
Now, let us check that condition \eqref{PhiDvc} implies condition \eqref{PhiaDu} under some (minor) additional assumptions. To this end, we will follow the same procedure as \cite[Theorem~2.15]{BMT} derived for a gradient field. Since we have for any $x\in Y_2$,
\[
\forall\,t\in\R,\quad {\partial \over \partial t}\big(v(Y(t,x))\big)=(\Phi\cdot\nabla v)(Y(t,x))\geq c>0,
\]
the mapping $t\mapsto v(Y(t,x))$ is a $C^1$-diffeomorphism on $\R$, hence there exists a unique $\tau(x)\in\R$ such that
\[
v\big(Y(\tau(x),x)\big) = 0,
\]
namely the trajectory $Y(\cdot,x)$ reaches the equipotential $\{v=0\}$ at time $\tau(x)$.
The uniqueness of $\tau$ combined with the semi-group property of the flow $Y$ easily implies that
\beq\label{tauY}
\forall\,x\in Y_2,\ \forall\,t\in\R,\quad \tau(Y(t,x))=\tau(x)-t.
\eeq
Moreover, by the implicit functions theorem and the $C^1(\R\times\R^2)$ regularity of the flow $Y$, the function $\tau$ belongs to $C^1(\R^2)$.
Then, define the positive function $\sigma_0\in C^1(\R^2)$ by
\beq\label{si0}
\sigma_0(x):=\exp\left(\int_0^{\tau(x)} ({\rm div}\,\Phi)(Y(s,x))\,ds\right)\quad\mbox{for }x\in\R^2.
\eeq
From now, assume that ${\rm div}\,(\Phi)\in C^1(\R^2)$. Then, the function $\sigma_0$ belongs to $C^1(\R^2)$.
By using \eqref{tauY} and the semi-group property of the flow $Y$, then making the change of variable $r=s+t$, we get that
\[
\forall\,x\in Y_2,\ \forall\,t\in\R,\quad \sigma_0(Y(t,x))=\exp\left(\int_t^{\tau(x)} ({\rm div}\,\Phi)(Y(r,x))\,dr\right).
\]
Next, taking the derivative of the previous equality at $t=0$, we obtain that
\[
\nabla\sigma_0\cdot\Phi+\sigma_0\,{\rm div}\,(\Phi)={\rm div}\,(\sigma_0\,\Phi)=0\;\;\mbox{in }\R^2,
\]
or equivalently, there exists a function $u_0\in C^1(\R^2)$ such that
\[
\Phi={1\over\sigma_0}\,R_\perp\nabla u_0\;\;\mbox{in }\R^2.
\]
This is nearly the desired condition \eqref{PhiaDu} except that the function $\sigma_0$ is not necessarily $\Z^2$-periodic.
However, also assuming that the function $\sigma_0$ is bounded from below and above by positive constants, the averaging procedure of \cite[Theorem~2.17]{BMT} allows us to build a positive periodic function $\sigma\in L^\infty_\sharp(Y_2)$ satisfying
\[
{\rm div}(\sigma\,\Phi)=0\;\;\mbox{in }\R^2,
\]
which is equivalent to condition \eqref{PhiaDu} with $a:=1/\sigma$. But the regularity of $\sigma$ is not ensured.
Finally, to get the regularity $\sigma\in C^1_\sharp(Y_2)$, it is enough to assume in addition that $\sigma_0$ and $\nabla\sigma_0$ are uniformly continuous in~$\R^2$ (see \cite[Remark~2.19]{BMT}).
\end{arem}
\par
In contrast with Corollary~\ref{cor.CbP} the following result uses both Theorem~\ref{thm.Cbsin} and the non-ergodic approach of \cite{Bri2} in any dimension, and provides an alternative approach to the two-dimensional Corollary~\ref{cor.CbP}.
\begin{acor}\label{cor.Cbdet}
For $d\geq 2$ and $n\in\N^*$, let $U_n=(u_1^n,u_2,\dots,u_d)$ be a sequence of vector fields in $C^2(Y_d)^d$ with $\nabla U_n\in C^1_\sharp(Y_d)^{d\times d}$, and let $a$ be a positive function in $C^1_\sharp(Y_d)$ such that
\beq\label{Un}
\left\{\ba{cl}
\det(\nabla U_n)>0 & \mbox{in }Y_d
\\ \ecart
\dis \rho_n:={1\over a\det(\nabla U_n)} & \mbox{converges uniformly in $Y_d$ to some }\rho\in C^1_\sharp(Y_d)
\\ \ecart
\rho\leq\rho_n &  \mbox{in }Y_d.
\ea\right.
\eeq
Then, the sequence of vector fields $(b_n)_{n\in\N}$ defined by
\beq\label{bnUn}
b_n=\rho_n\,\Phi,\quad\mbox{where}\quad \Phi:=\left\{\ba{ll}
a\,R_\perp\nabla u_2 & \mbox{if }d=2
\\ \ecart
a\,(\nabla u_2\times\cdots\times\nabla u_d) & \mbox{if }d>2
\ea\right.
\eeq
converges uniformly in $Y_d$ to the vector field $b=\rho\,\Phi$. Moreover, we have:
\begin{itemize}
\item If $\rho$ is positive in $Y_d$, then the flow $X$ associated with $b$ satisfies the asymptotics
\beq\label{asyDuk}
\forall\,x\in Y_d,\quad \lim_{t\r\infty} \frac{X(t,x)}{t}=\left\{\ba{ll}
\dis \underline{a\rho}\int_{Y_2}R_\perp\nabla u_2(y)\,dy & \mbox{if }d=2
\\ \ecart
\dis \underline{a\rho}\int_{Y_d}(\nabla u_2\times\cdots\times\nabla u_d)(y)\,dy & \mbox{if }d>2,
\ea\right.
\eeq
where $\underline{a\rho}$ is the harmonic mean of $\rho$.
\item If $\rho$ vanishes in $Y_d$ and $\underline{a\rho}=0$, then the flow $X$ associated with $b$ satisfies the asymptotics
\beq\label{asyDuk0}
\forall\,x\in Y_d,\quad \lim_{t\r\infty} \frac{X(t,x)}{t}=0_{\R^d}.
\eeq
\item If $\rho$ vanishes in $Y_d$ and $\underline{a\rho}>0$, then we get the set
\beq\label{CbDuk}
\sfC_b=[0_{\R^d},\zeta]\quad\mbox{with}\quad \zeta:=\left\{\ba{ll}
\dis \underline{a\rho}\int_{Y_2}R_\perp\nabla u_2(y)\,dy\neq 0_{\R^2} & \mbox{if }d=2
\\ \ecart
\dis \underline{a\rho}\int_{Y_d}(\nabla u_2\times\cdots\times\nabla u_d)(y)\,dy\neq 0_{\R^d} & \mbox{if }d>2.
\ea\right.
\eeq
\end{itemize}
\end{acor}
\begin{proof}{ of Corollary~\ref{cor.Cbdet}}
We have
\beq\label{bnDun1}
b_n\cdot\nabla u_1^n=\left\{\ba{ll}
a\rho_n\,\nabla u_1^n\cdot R_\perp\nabla u_2  & \mbox{if }d=2
\\ \ecart
a\rho_n\,\nabla u_1^n\cdot(\nabla u_2\times\cdots\times\nabla u_d) & \mbox{if }d>2
\ea\right\}
=a\rho_n\,\det(\nabla U_n)=1\quad\mbox{in }Y_d,
\eeq
and
\[
{1\over a\rho_n}\,b_n=\left\{\ba{ll}
R_\perp\nabla u_2  & \mbox{if }d=2
\\ \ecart
\nabla u_2\times\cdots\times\nabla u_d & \mbox{if }d>2
\ea\right.
\]
is divergence free in $\R^d$.
Hence, from \cite[Corollary 4.1]{Bri2} we deduce that $\sfC_{b_n}$ is the singleton $\{\zeta_n\}$ with
\beq\label{zetanDuk}
\zeta_n:={\dis \int_{Y_d}(a\rho_n)^{-1}(y)\,b_n(y)\,dy\over\dis \int_{Y_d}(a\rho_n)^{-1}(y)\,dy}=
\left\{\ba{ll}
\dis =\underline{a\rho_n}\int_{Y_2}R_\perp\nabla u_2(y)\,dy & \mbox{if }d=2
\\ \ecart
\dis =\underline{a\rho_n}\int_{Y_d}(\nabla u_2\times\cdots\times\nabla u_d)(y)\,dy & \mbox{if }d>2,
\ea\right.
\eeq
where $\underline{a\rho_n}$ is the harmonic mean of $a\rho_n$.
\par\smallskip\noindent
Let us conclude:
\begin{itemize}
\item If $\rho>0$ in $Y_d$, then the sequence $(a\rho_n)^{-1}$ converges uniformly to $(a\rho)^{-1}$ in $Y_d$. Therefore, by the first case of Theorem~\ref{thm.Cbsin}
combined with \eqref{zetanDuk} we get that $\sfC_b=\{\zeta\}$, or equivalently by \eqref{Cbsin}, asymptotics \eqref{asyDuk} holds.
\item Otherwise, $\rho$ vanishes in $Y_2$. Moreover, by \eqref{bnUn} the vector field $a^{-1}\Phi$ is clearly divergence free in~$\R^2$.
Therefore, by virtue of Corollary~\ref{cor.siPhidiv0} with $\sigma=a^{-1}$, we deduce the null asymptotics \eqref{asyDuk0} if $\underline{a\rho}=0$,
and the set $\sfC_b$ \eqref{CbDuk} if $\underline{a\rho}>0$.
It remains to prove that $\zeta\neq 0_{\R^d}$ in \eqref{CbDuk}. By the definition of $U_n$ and \eqref{Un} we have
\[
\int_{Y_2}\underbrace{\det(\nabla U_n)(y)}_{>0}dy=\left\{\ba{ll}
\dis \int_{Y_2}\nabla u_1^n(y)\cdot R_\perp\nabla u_2(y)\,dy>0  & \mbox{if }d=2
\\ \ecart
\dis \int_{Y_d}\nabla u_1^n(y)\cdot(\nabla u_2\times\cdots\times\nabla u_d)(y)>0 & \mbox{if }d>2.
\ea\right.
\]
Hence, from the quasi-affinity of the determinant (see, {\em e.g.}, \cite[Section~4.3.2]{Dac}), namely:
\[
\det\left(\int_{Y_2}\nabla U_n(y)\,dy\right)=\int_{Y_2}\det(\nabla U_n)(y)\,dy>0,
\]
we deduce that
\[
\left\{\ba{ll}
\dis \int_{Y_2}\nabla u_1^n\cdot\left(\int_{Y_2}R_\perp\nabla u_2\right)=\int_{Y_2}\nabla u_1^n\cdot R_\perp\nabla u_2>0 & \mbox{if }d=2
\\ \ecart
\dis \int_{Y_d}\nabla u_1^n\cdot\left(\int_{Y_d}\nabla u_2\times\cdots\times\int_{Y_d}\nabla u_d\right)
=\int_{Y_d}\nabla u_1^n\cdot(\nabla u_2\times\cdots\times\nabla u_d)>0 & \mbox{if }d>2.
\ea\right.
\]
Therefore, again using the quasi-affinity of the determinant (multiplying the second equality by any constant vector of $\R^d$ to get a determinant) we get that
\[
\left\{\ba{ll}
\dis \int_{Y_2}R_\perp\nabla u_2(y)\,dy\neq 0_{\R^2} & \mbox{if }d=2
\\ \ecart
\dis \int_{Y_d}(\nabla u_2\times\cdots\times\nabla u_d)(y)\,dy=\int_{Y_d}\nabla u_2(y)\,dy\times\cdots\times\int_{Y_d}\nabla u_d(y)\,dy\neq 0 & \mbox{if }d>2,
\ea\right.
\]
which implies that $\zeta\neq 0_{\R^d}$ in \eqref{CbDuk}.
\end{itemize}
\end{proof}
\begin{arem}\label{rem.corPdet}
In Corollary~\ref{cor.CbP} and in the two-dimensional case of Corollary~\ref{cor.Cbdet}, the vector field $b_n$ has the the same form $b_n=\rho_n\,a\,R_\perp\nabla u$.
In Corollary~\ref{cor.CbP} the function $\rho_n$ is arbitrary, while $\nabla u$ satisfies the ergodic condition~\eqref{ergDu}. On the contrary, in the two dimensional case of Corollary~\ref{cor.Cbdet} $\rho_n$ does depend on the functions $a$ and $\nabla u$ by \eqref{Un}, while $\nabla u$ is arbitrary.
Therefore, these results provide two quite different approaches on the asymptotics of the flow: an ergodic one using \cite{Pei} and a non-ergodic one using \cite{Bri2}.
\end{arem}
%%%%%%%%%%
\section{Examples}\label{s.exa}
The following counter-example shows that Theorem~\ref{thm.Cbsin} does not extend to a vector-valued perturbation.
\begin{Aexa}\label{exa.1}
Consider the fields $b$ and $b_n$ defined by
\[
b(x)=a(x)\,e_1\quad\mbox{and}\quad b_n(x)=a_n(x)\,(e_1+\gamma_n e_2),\qquad \mbox{for }x\in Y_2\mbox{ and }n\in\N,
\]
where $a$ is a non negative function in $C^1_\sharp(Y_2)$, $a_n:=a+1/n$, and $(\gamma_n)_{n\in\N}$ is a positive sequence in $\R\setminus\Q$ which converges to $0$.
In this case, the flows associated with the vector fields $b$ and $b_n$ can be computed explicitly.
\par
On the one hand, by the ergodic case of \cite[Section~3.1]{Tas} with the irrational rotation number~$\gamma_n$, the flow $X_n$ associated with the field $b_n$ satisfies the asymptotics
\[
\forall\, x\in Y_2,\quad\lim_{t\to\infty}\,{X_n(t,x)\over t}=\underline{a_n}\,(e_1+\gamma_n\,e_2),
\]
where $\underline{a_n}$ is the harmonic mean of $a_n$, or equivalently by \eqref{Cbsin},
\beq\label{zetanan}
\sfC_{b_n}=\{\zeta_n\}\quad\mbox{with}\quad \zeta_n:=\underline{a_n}\,(e_1+\gamma_n\,e_2).
\eeq
\par
On the other hand, by the non-ergodic case of \cite[Section~3.1]{Tas} with the rational rotation number $0$ and its extension when $a$ vanishes in $Y_2$, the flow $X$ associated with the field $b$ satisfies the asymptotics
\[
\forall\,x\in Y_2,\quad \lim_{t\to\infty} \frac{X(t,x)}{t}=\left\{\ba{cl}
\dis \underline{a(\cdot \,e_1+x)}\,e_1 & \mbox{if $a(\cdot \,e_1+x)$ is positive in }Y_2
\\ \ecart
0 & \mbox{if $a(\cdot \,e_1+x)$ vanishes in }Y_2,
\ea\right.
\]
where $\underline{a(\cdot \,e_1+x)}$ is the harmonic mean defined by
\[
\underline{a(\cdot \,e_1+x)}:=\left(\int_{Y_1}{dt\over a(t\,e_1+x)}\right)^{-1}\quad\mbox{for }x\in Y_2,
\]
or equivalently, the set $\sfA_b$ \eqref{Ab} is given by
\[
\sfA_b=\left\{\underline{a(\cdot \,e_1+x)}\,e_1:x\in Y_2\right\}.
\]
Moreover, the function $\big(x\mapsto\underline{a(\cdot \,e_1+x)}\big)$ is continuous on the compact set~$Y_2$.
It is clear at any point $x\in Y_2$ such that $a(\cdot \,e_1+x)$ is positive. Otherwise, if $a(\cdot \,e_1+x)\in C^1_\sharp(Y_1)$ vanishes in~$Y_1$, by Fatou's lemma we get that for any sequence $(x_n)_{n\in\N}$ converging to $x$,
\[
\infty={1\over \underline{a(\cdot \,e_1+x)}}\leq \liminf_{n\to\infty}\left({1\over \underline{a(\cdot \,e_1+x_n)}}\right)
=\lim_{n\to\infty}\left({1\over \underline{a(\cdot \,e_1+x_n)}}\right)=\infty.
\]
Hence, the set $\sfA_b$ is actually a closed line segment of $\R^2$, which by \eqref{ACb} implies that
\beq\label{Cb=Ab}
\sfC_b={\rm conv}(\sfA_b)=\sfA_b=\left\{\underline{a(\cdot \,e_1+x)}\,e_1:x\in Y_2\right\}.
\eeq
\par\noindent
In particular, when $a$ vanishes in $Y_2$, we get that
\[
\sfC_b=[0_{\R^2},\zeta]\quad\mbox{with}\quad \zeta:=\left(\max_{x\in Y_2}\,\underline{a(\cdot \,e_1+x)}\right)e_1.
\]
Therefore, taking into account \eqref{zetanan}, contrary to the second case of Theorem~\ref{thm.Cbsin} we may have
\[
\underline{a}\,e_1=\lim_{n\to\infty}\zeta_n\neq \zeta=\left(\max_{x\in Y_2}\,\underline{a(\cdot \,e_1+x)}\right)e_1.
\]
For example, take
\[
a(x):=\sin^2(\pi x_1)+\sin^2(\pi x_2)\quad\mbox{for }x=(x_1,x_2)\in\R^2.
\]
Then, it follows that
\[
\ba{rll}
\dis {1\over \underline{a}}=\int_{Y_1}\left(\int_{Y_1}{dx_2\over \sin^2(\pi x_1)+\sin^2(\pi x_2)}\right)dx_1 & > &
\dis \int_{Y_1}{dx_1\over \sin^2(\pi x_1)+1}
\\ \ecart
(\mbox{for any $x_1\in Y_1$ and $x_2={1\over 2}$}) & = & \dis \min_{x\in Y_2}\left(\int_{Y_1}{dt\over \sin^2(\pi t+\pi x_1)+\sin^2(\pi x_2)}\right)
\\ \ecart
& = & \dis \min_{x\in Y_2}\left({1\over \underline{a(\cdot \,e_1+x)}}\right),
\ea
\]
which implies that
\[
\underline{a}<\max_{x\in Y_2}\,\underline{a(\cdot \,e_1+x)}.
\]
Therefore, Theorem~\ref{thm.Cbsin} does not extend in general to the case where the direction $\Phi$ of the field $b_n=\rho_n\,\Phi$ also depends on $n$.
\par
Finally, note that the inclusion $\{0_{\R^2},\zeta\}\subset \sfA_b$ of the second case of Theorem~\ref{thm.Cbsin} is not in general an equality, since in the particular case \eqref{Cb=Ab} $\sfA_b$ is the closed line segment $[0_{\R^2},\zeta]$.
\end{Aexa}
\par\medskip
The second example shows that the singleton condition is not in general asymptotically preserved under the assumptions of  Theorem~\ref{thm.Cbsin}.
\begin{Aexa}\label{exa.2}
Let $\nabla u\in C^1_\sharp(Y_2)^2$ be satisfying the ergodic condition \eqref{ergDu}, let $(\rho_n)_{n\in\N}$ be the sequence of positive functions in $C^1_\sharp(Y_2)$ defined by
\[
\rho_n(x):=\left(\sin^2(\pi x_1)+\sin^2(\pi x_2)+1/n\right)^{\al}\quad\mbox{for }x\in Y_2,\quad\mbox{with }\al\in(1/2,1).
\]
and let $(b_n)_{n\in\N}$ be the sequence of vector fields defined by $b_n=\rho_n\,R_\perp\nabla u$. Since the function $(t\mapsto t^\alpha)$ is uniformly continuous in $[0,\infty)$, the sequence $(\rho_n)_{n\in\N}$ converges uniformly in $Y_d$ to the function
\[
\rho(x):=\left(\sin^2(\pi x_1)+\sin^2(\pi x_2)\right)^{\al}\quad\mbox{for }x=(x_1,x_2)\in Y_2,
\]
which belongs to $C^1_\sharp(Y_2)$ due to $\al>1/2$, and vanishes at the sole point $(0,0)$ in the torus $Y_2$.
Moreover, we have for any $x$ close to $(0,0)$,
\[
{c^{-1}\over |x|^{2\alpha}}\leq{1\over \rho(x)}\leq {c\over |x|^{2\alpha}},\quad\mbox{for some }c>1,
\]
so that $\underline{\rho}>0$ due to $\alpha<1$ (see Remark~\ref{rem.hrho}).
Therefore, by the asymptotics \eqref{CbPhi} of Corollary~\ref{cor.CbP} and condition \eqref{ergDu} we get that
\[
\sfC_b=[0,\zeta]\quad\mbox{with}\quad
\zeta:=\left(\int_{Y_2}{dx\over\left(\sin^2(\pi x_1)+\sin^2(\pi x_2)\right)^{\al}}\right)^{-1}\kern -.4em\int_{Y_d}R_\perp\nabla u(y)\,dy\neq 0.
\]
\par
Note that, if $\al\geq 1$, then $\underline{\rho}=0$. Therefore, by Corollary~\ref{cor.CbP} we get that $\sfC_b=\{0_{\R^2}\}$, and the flow $X$ satisfies the null asymptotics \eqref{asybsi0}.
\end{Aexa}
\par\medskip
The third example illustrates Corollary~\ref{cor.Cbdet}.
\begin{Aexa}\label{exa.3}
Let $U_n=(u_1^n,u_2,\dots,u_d)\in C^2(Y_d)^{d}$, $d\geq 2$ and $n\geq 1$, be such that $\nabla U_n$ is $\Z^d$-periodic, the functions $u_2,\dots,u_d$ only depend on the variables $x'=(x_2,\dots,x_d)$,
\[
u_1^n(x):=\int_0^{x_1}{dt\over f_n(t,x')}\quad\mbox{and}\quad\Delta(x'):=\det\left(\left[{\partial u_i\over\partial x_j}(x')\right]_{2\leq i,j\leq d}\right)>0
\quad\mbox{for }x\in Y_d,
\]
where $(f_n)_{n\in\N}$ is a positive sequence in $C^1_\sharp(Y_d)^\N$ which converges uniformly to $f\leq f_n$ in $Y_d$.
Expanding the determinant with respect to its first column we have
\[
\forall\,x\in Y_d,\quad\det(\nabla U_n)(x)={\Delta(x')\over f_n(x)}>0\quad\mbox{and}\quad \rho_n(x):={1\over \det(\nabla U_n)(x)}\to \rho(x):={f(x)\over \Delta(x')}\leq\rho_n(x)
\]
uniformly in $Y_d$, so that condition \eqref{Un} is fulfilled with $a=1$.
Define the vector field $b_n$ in $C^1_\sharp(Y_d)^d$ by \eqref{bnUn}.
Therefore, the sequence $(b_n)_{n\in\N}$ converges uniformly in $Y_d$ to the function $b$ given by
\[
b(x)=\left\{\ba{ll}
\dis {f(x)\over \Delta(x')}\,R_\perp\nabla u_2(x) & \mbox{if }d=2
\\ \ecart
\dis {f(x)\over \Delta(x')}\,(\nabla u_2\times\cdots\times \nabla u_d)(x) & \mbox{if }d>2,
\ea\right.
\quad\mbox{for }x\in Y_d.
\]
Moreover, due to the $1$-periodicity of $\nabla_{x'}\,u_1^n$ with respect to the variable $x_1$, we have
\[
\forall\,x'\in\R^{d-1},\quad \int_0^1\nabla_{x'}\left({1\over f_n(t,x')}\right)dt=0,
\]
which implies the existence of a positive constant $c_n$ such that
\beq\label{cn}
\forall\,x'\in\R^{d-1},\quad \int_0^1{dt\over f_n(t,x')}=c_n.
\eeq
Hence, from inequality $f\leq f_n$ and Fatou's lemma we deduce that
\[
\forall\,x'\in\R^{d-1},\quad \limsup_{n\to\infty}c_n\leq \int_0^1{dt\over f(t,x')}\leq\liminf_{n\to\infty}\int_0^1{dt\over f_n(t,x')}=\liminf_{n\to\infty}c_n,
\]
which implies that
\beq\label{fncn}
\forall\,x'\in\R^{d-1},\quad\int_0^1{dt\over f(t,x')}= \lim_{n\to\infty}c_n.
\eeq
Then, we have the following alternative:
\begin{itemize}
\item If $f$ is positive in $Y_d$, then by virtue of the first case of Corollary~\ref{cor.Cbdet} we obtain the asymptotics of the flow associated with $b$
\[
\forall\,x\in Y_d,\quad \lim_{t\r\infty} \frac{X(t,x)}{t}=\left\{\ba{ll}
\dis \underline{\rho}\int_{Y_2}R_\perp\nabla u_2(y)\,dy & \mbox{if }d=2
\\ \ecart
\dis \underline{\rho}\int_{Y_d}(\nabla u_2\times\cdots\times\nabla u_d)(y)\,dy & \mbox{if }d>2,
\ea\right.
\]
where
\[
\underline{\rho}=\left(\int_{Y_d}{\Delta(x')\over f(x)}\,dx\right)^{-1}.
\]
\item If $f$ vanishes at some point $x_0\in Y_d$, then since $f(\cdot,x_0')$ is in $C^1_\sharp(Y_1)$ and vanishes at $t=(x_0)_1$, we have
by \eqref{fncn}
\beq\label{hf}
\forall\,x'\in\R^{d-1},\quad \lim_{n\to\infty}c_n=\lim_{n\to\infty}\int_0^1{dt\over f_n(t,x')}=\int_0^1{dt\over f(t,x_0')}=\infty.
\eeq
Assume by contradiction that $\underline{f}>0$.
Then, using successively Lebesgue's theorem with inequality $f\leq f_n$, Fubini's theorem and equality \eqref{cn}, we get that
\[
\infty>\int_{Y_d}{dx\over f(x)}=\lim_{n\to\infty}\int_{Y_d}{dx\over f_n(x)}=\lim_{n\to\infty}\left[\int_{Y_{d-1}}dx'\left(\int_0^1{dt\over f_n(t,x')}\right)\right]=\lim_{n\to\infty}c_n,
\]
which contradicts \eqref{hf}.
Hence, we deduce that $\underline{f}=0$, and due to the positivity of $\Delta$ we get that
\[
\underline{\rho}=\left(\int_{Y_d}{\Delta(x')\over f(x)}\,dx\right)^{-1}=0.
\]
Therefore, by virtue of the second case of Corollary~\ref{cor.Cbdet} we obtain the null asymptotics~\eqref{asyDuk0}.
\end{itemize}
Note that the third case of Corollary~\ref{cor.Cbdet} cannot arise when the functions $u_2,\dots,u_d$ are independent of the variable $x_1$.
\end{Aexa}
\par
The fourth example deals with the case of an electric field. It is based on the divergence-curl Proposition~\ref{pro.divcurl}, and illustrates the framework of Theorem~\ref{thm.Cbsin}. We cannot characterize precisely the set $\sfC_b$ except in the two-dimensional ergodic case.
However, the two-dimensional case and the three-dimensional case are shown to be quite different.
\begin{Aexa}\label{exa.4}
Let $\sigma\in C^3_\sharp(Y_d)$ be a positive function with $\int_{Y_d}\sigma(y)\,dy=1$. Consider the vector-valued function $U\in C^2(\R^d)^d$ (see, {\em e.g.}, \cite[Theorem~8.13]{GiTr}) unique  solution (up to an additive constant vector) to the conductivity problem
 \beq\label{U}
 \left\{\ba{ll}
{\rm Div}(\sigma DU)=0_{\R^d} & \mbox{in }\R^d
 \\ \ecart
 y\mapsto U(y)-y & \mbox{is $\Z^d$-periodic},
 \ea\right.
 \eeq
where $DU=(\nabla U)^T$ and the vector-valued operator ${\rm Div}$ consists in the divergence of the columns of $\sigma DU$.
The variational formulation of \eqref{U} reads as
\beq\label{DU}
DU\in C^1_\sharp(Y_d)^d\quad\mbox{and}\quad\forall\,\Psi\in C^1_\sharp(Y_d)^d,\;\;\int_{Y_d}\sigma(y) DU(y):D\Psi(y)\,dy=0.
\eeq
In \eqref{DU} $``:"$ denotes the scalar product in $\R^{d\times d}$ defined by
\[
M:N:={\rm tr}(M^TN)\quad\mbox{for }M,N\in\R^{d\times d}.
\]
The so-called homogenized matrix (see, {\em e.g.}, \cite[Chapter~I,~Section~2.3]{BLP}) associated with the conductivity $\sigma$ is defined by
 \beq\label{A*}
 A^*:=\int_{Y_d}\sigma(y)\,DU(y)\,dy,
 \eeq
 which is known to be symmetric positive definite.
 Also define the associated electric field
 \beq\label{bDula}
 b_\lambda:=\nabla u_\lambda=\nabla(U\lambda)=DU\lambda\quad\mbox{for }\lambda\in\R^d\setminus\{0_{\R^d}\}.
 \eeq
 \par\noindent
 {\it Case $d=2$:}
 Alessandrini and Nesi \cite[Theorems~1,2]{AlNe} have proved that $U$ is a $C^1$-diffeomorphism of $\R^2$ with
 \beq\label{detDU}
 \det(DU)>0\quad\mbox{in }Y_2.
 \eeq
Let $\la\in\R^2\setminus\{0_{\R^2}\}$. As a consequence (see \cite[Proposition~2]{AlNe}) the gradient field $b_\lambda$ defined by \eqref{bDula} does not vanish in $Y_2$. Hence, the set of invariant probability measures $\cI_{b_\lambda}$ for the flow associated with $b_\lambda$ does not contain any Dirac measure.
\par
Moreover, by \eqref{DU} and  \eqref{bDula} we have
\beq\label{blasi}
\forall\,\psi\in C^1_\sharp(Y_2),\quad \int_{Y_2}b_\lambda(y)\cdot\nabla\psi(y)\,\sigma(y)\,dy
=\int_{Y_2}\sigma(y)\nabla u_\lambda(y)\cdot\nabla\psi(y)\,dy=0,
\eeq
which by virtue of the equivalence $(i)$-$(iii)$ of Proposition~\ref{pro.divcurl} implies that $\sigma(x)\,dx$ is an invariant probability measure with positive density.
Therefore, by definition~\eqref{A*} we can only conclude that (recall that $A^*$ is positive definite and $\la$ is non zero)
\beq\label{A*la}
A^*\lambda=\int_{Y_2}b_\lambda(y)\,\sigma(y)\,dy\in \sfC_{b_\lambda}.
\eeq
On the other hand, let $\mu\in\cI_{b_\lambda}$ be an invariant probability measure for the flow $X_\lambda$ associated with $b_\lambda$.
By the divergence-curl relation \eqref{divcurl} and \eqref{U} we have
\beq\label{Dula}
\int_{Y_2}|b_\lambda(y)|^2\,\mu(dy)=\int_{Y_2}b_\lambda(y)\cdot\nabla u_\lambda\,d\mu(y)
=\left(\int_{Y_2}b_\lambda(y)\,d\mu(y)\right)\cdot\lambda.
\eeq
Due to inequality \eqref{detDU} the gradient field $b_\lambda=DU\lambda$ does not vanish in $Y_2$.
Hence, since $\mu$ is a probability measure, we deduce that $\mu(\{b_\lambda\neq 0\})>0$, which implies that the first term of \eqref{Dula} is positive.
Therefore, we get that
\[
0_{\R^2}\neq\int_{Y_2}b_\lambda(y)\,d\mu(y)\in \sfC_{b_\lambda}.
\]
This combined with \eqref{A*la} yields
\beq\label{Cbla2}
\{A^*\lambda\}\subset \sfC_{b_\lambda}\subset\R^2\setminus\{0_{\R^2}\}.
\eeq
We may improve the former result under the extra ergodic condition
\beq\label{ergA*la}
\forall\,\kappa\in\Z^2\setminus\{0_{\R^2}\},\quad (A^*\lambda)\cdot\kappa\neq 0.
\eeq
Indeed, since by \eqref{blasi} $\sigma\,b_\lambda$ is divergence free, by a classical duality argument there exists a potential $v_\lambda$ with $\nabla v_\lambda\in C^1_\sharp(Y_2)^2$, such that
$\sigma\,b_\lambda=R_\perp\nabla v_\lambda$ in $Y_2$.
Moreover, since by \eqref{A*la}
\[
\int_{Y_2}\nabla v_\lambda(y)\,dy=-\,R_\perp\int_{Y_2}\sigma(y)\,b_\lambda(y)\,dy=-\,R_\perp A^*\lambda,
\]
condition \eqref{ergA*la} means that the gradient field $\nabla v_\lambda$ satisfies the ergodic condition \eqref{ergDu}.
Hence, the vector field $b_\lambda=1/\sigma\,R_\perp\nabla v_\lambda$ satisfies the first result of Corollary~\ref{cor.CbP} with $u:=v_\lambda$, $a:=1/\sigma$ and $\rho:=1$.
Therefore, the flow $X_\lambda$ associated with $b_\lambda$ satisfies asymptotics~\eqref{asybsi} which reads as
\[
\forall\,x\in Y_2,\quad \lim_{t\r\infty} \frac{X_\lambda(t,x)}{t}=\underline{1/\sigma}\int_{Y_2}R_\perp\nabla v(y)\,dy=A^*\lambda,
\]
or equivalently by \eqref{Cbsin}, we obtain that
\beq\label{ergCbla2}
\sfC_{b_\lambda}=\{A^*\lambda\}.
\eeq
\noindent
 {\it Case $d=3$:} Contrary to the two-dimensional case with the positivity \eqref{detDU}, by virtue of \cite[Theorem~4.1]{BMN} there exists a positive conductivity $\sigma_\gamma\in L^\infty_\sharp(Y_3)$ such that the vector-valued function $U_\gamma$ solution to the equation \eqref{U} with $\sigma_\gamma$ has a determinant which changes sign.
\par
More precisely, the conductivity of \cite{BMN} rescaled by its $Y_3$-average value denoted by $\sigma_\gamma$ (recall that $\sigma_\gamma(x)\,dx$ has to be a probability measure) takes two values: $\sigma_\gamma=1$ in a cubic symmetric lattice of interlocking rings which do not intersect, and $\sigma_\gamma=\gamma\ll 1$ elsewhere in $\R^3$. The conductivity $\sigma_\gamma$ is thus not regular. However, enlarging each ring with a width $d_\gamma\ll 1$, we can build a new conductivity $\sigma\in C^3_\sharp(Y_3)$ (also depending on $\gamma$) whose values pass from $1$ on the boundary of each ring to $\gamma$ on the boundary of the corresponding enlarged ring. Then, it is easy to check that for $\gamma$ and $d_\gamma$ small enough, the matrix-valued fonction $DU$ defined by \eqref{U} with the regular conductivity $\sigma$ has a determinant which also changes sign.
Thus, by a continuity continuity argument we get that
\beq\label{detDU=0}
\exists\,y_0\in Y_3,\quad \det(DU)(y_0)=0,
\eeq
which implies that there exists $\lambda\in\R^3\setminus\{0_{\R^3}\}$ such that $DU(y_0)\lambda=0_{\R^3}$.
Hence, the gradient field $b_\lambda$ defined by \eqref{bDula} vanishes at point $y_0$.
Therefore, in contrast with the two-dimensional result \eqref{A*la} we obtain the more complete result
\beq\label{Ala}
A^*\lambda=\int_{Y_3}b_\lambda(y)\,\sigma(y)\,dy\in \sfC_{b_\lambda}\setminus\{0_{\R^3}\}\quad\mbox{and}\quad [0_{\R^3},A^*\lambda]\subset \sfC_{b_\lambda},
\eeq
since the Dirac mass $\delta_{y_0}$ belongs to $\cI_{b_\delta}$ and
\[
\int_{Y_3}b_\lambda(y)\,d\delta_{y_0}(dy)=b_\lambda(y_0)=0_{\R^3}.
\]
Result \eqref{Ala} corresponds to the second case of Theorem~\ref{thm.Cbsin}.
In contrast with the result \eqref{Cbla2} of the two-dimensional case, we obtain that
\beq\label{Cbla3}
[0_{\R^3},A^*\lambda]\subset \sfC_{b_\lambda}.
\eeq
\end{Aexa}
%%%%%%%%%%
\section{Homogenization of linear transport equations}\label{s.hom}
The following theorem is an extension of various homogenization results \cite{Bre,Gol1,Gol2,HoXi,Tas} (and the references therein) of linear transport equations with an oscillating velocity, which are based on the classical ergodic approach. Here, in a regular and periodic framework the ergodic approach is replaced by the singleton approach of Section~\ref{ss.appsin}, whose a very particular case has been first obtained in \cite[Corollary~4.4]{Bri1}.
\begin{atheo}\label{thm.hom}
Let $b$ be a vector field in $C^1_\sharp(Y_d)^d$ and let $u_0\in C^1(\R^d)$.
Consider the transport equation with the oscillating velocity $b(x/\ep)$:
\beq\label{eque}
\left\{\ba{ll}
\dis {\partial u_\ep\over\partial t}-b(x/\ep)\cdot\nabla u_\ep=0 & \mbox{in }(0,\infty)\times\R^d
\\ \ecart
u_\ep(0,x)=u_0(x) & \mbox{for }x\in\R^d.
\ea\right.
\eeq
Assume that there exists a vector $\zeta\in\R^d$ such that
\beq\label{asyze}
\forall\,x\in Y_d,\quad\lim_{t\to\infty}\,{X(t,x)\over t}=\zeta,
\eeq
where $X$ is the flow~\eqref{X} associated with the vector field $b$.
Then, the solution $u_\ep$ to transport equation \eqref{eque} converges strongly in $L^p_{\rm loc}([0,\infty)\times\R^d)$ for any $p\in[1,\infty)$, to $u_0(x+t\,\zeta)$ which is solution to the transport equation \eqref{eque} with the constant velocity $\zeta$ in place of $b(x/\ep)$.
\end{atheo}
\begin{proof}{ of Theorem~\ref{thm.hom}}
Let $X$ be the flow associated to the vector field $b$. By $\ep$-rescaling the flow $X$, let us define the flow $X_\ep$ associated with the oscillating vector field $b(x/\ep)$ by
\beq\label{Xep}
\forall\,(t,x)\in (0,\infty)\times Y_d,\quad X_\ep(t,x):=\ep\,X\left({t\over\ep},{x\over\ep}\right)
=x+\ep\int_0^\frac{t}\varepsilon b\left(X\big(s,\frac{x}{\varepsilon}\big)\right)ds.
\eeq
Taking into account the regularity conditions the characteristics method induced by the flow $X_\ep$ implies that the solution $u_\ep$ to \eqref{eque} is given by
\beq\label{uepb}
\forall\,(t,x)\in [0,\infty)\times Y_d,\quad u_\ep(t,x)=u_0(X_\ep(t,x))=u_0\left(x+\ep\int_0^\frac{t}\varepsilon b\left(X\big(s,\frac{x}{\varepsilon}\big)\right)ds\right).
\eeq
On the other hand, let $(\varepsilon_n)_{n\in\N}$ be a positive sequence converging to $0$. Let $(t,x)\in (0,\infty)\times Y_d$, set $t_n :=t/\varepsilon_n$ and $x_n:=x/\varepsilon_n$. Then, by virtue of Proposition~\ref{pro.invmeas} the limit points of the sequence
\[
v_n := \varepsilon_n \int_0^{\frac{t}{\varepsilon_n}} b\left(X\big(s,\frac{x}{\varepsilon_n}\big)\right)ds =t\times \frac{1}{t_n} \int_0^{t_n} b(X(s,x_n))\,ds
\]
belong to $t\,\sfC_b$. However, by \eqref{asyze} combined with equivalence \eqref{Cbsin} we have $\sfC_b=\{\zeta\}$. Hence, for any positive sequence $(\varepsilon_n)_{n\in\N}$ converging to $0$, the whole sequence $(v_n)_{n\in\N}$ converges to $t\,\zeta$, which combined with \eqref{Xep} implies that
\beq\label{conXep}
\forall\,(t,x)\in (0,\infty)\times Y_d,\quad \lim_{\ep\to 0}\,X_\ep(t,x)=x+t\,\zeta.
\eeq
Moreover, making the change of variable $r=\ep\,s$ in \eqref{Xep} we have
\[
\forall\,(t,x)\in [0,\infty)\times Y_d,\quad X_\ep(t,x)=x+\int_0^t b\left(X\big(\frac{r}{\varepsilon},\frac{x}{\varepsilon}\big)\right)dr.
\]
This combined with the boundedness of $b$ and Lebesgue's theorem implies that the pointwise convergence \eqref{conXep} of $X_\ep$ holds actually in $L^p_{\rm loc}([0,\infty)\times\R^d)$ for any $p\in[1,\infty)$.
Therefore, by the expression \eqref{uepb} with $u_0\in C^1(\R^d)$, $u_\ep(t,x)$ converges strongly in $L^p_{\rm loc}([0,\infty)\times\R^d)$ for any $p\in[1,\infty)$, to the function $u_0(x+t\,\zeta)$, which concludes the proof.
\end{proof}
%%%%%%%%%%
\appendix
\section{Derivation of invariant probability measures}
Let $T_t$ for $t\in\R$, be the mapping from $C^0_\sharp(Y_d)$ into itself defined by
\beq\label{Tt}
(T_t f)(x):=f\big(X(t,x)\big)\quad\mbox{for }f\in C^0_\sharp(Y_d)\mbox{ and }x\in Y_d.
\eeq
When a flow preserves the set of the continuous functions on a compact metric space, the existence of an invariant probability measure  for the flow is a classical statement which can be derived thanks to a weak compactness argument applied to sequences of probability measures defined from the Birkhoff time averages in (1.5) (see, {\em e.g.}, \cite[Theorem~1, Section~1.8]{CFS} in the discrete time case).
The following result adapts this statement restricting it to the limit points of the Birkhoff time averages for a given fixed function, adding possible variations of the spatial parameter $x$ in the averages.
\begin{apro}\label{pro.invmeas}
Let $b\in C^1_\sharp(Y_d)^d$. There exists an invariant {probability measure on $Y_d$} for the flow $X$ \eqref{X} associated with $b$.
Moreover, let $g\in C^0_\sharp(Y_d)$, let $(x_n)_{n\in\N}\in(\R^d)^\N$, and let $(t_n)_{n\in\N}\in\R^\N$ be such that $\lim_n t_n = \infty$. Then, for any limit point $a$ of the sequence $(u_n)_{n\in\N}\in\R^\N$ defined by
\beq\label{ung}
u_n := \frac{1}{t_n}\int_0^{t_n} g(X(s,x_n))\,ds,\quad n\in\N,
\eeq
there exists a probability measure $\mu\in \cM_p(Y_d)$ (depending {\em a priori} on $(x_n)_{n\in\N}$ and $g$) which is invariant for the flow $X$ and which satisfies
\beq\label{afmu}
a=\int_{Y_d} g(y)\, d\mu(y).
\eeq
\end{apro}
\begin{proof}{ of Proposition~\ref{pro.invmeas}}
We will use the following result which is proved below.
\begin{alem} \label{lem-extrac}
Let $(y_n)_{n\in\N}\in(\R^d)^\N$, let $(r_n)_{n\in\N}\in\R^\N$ be such that $\lim_n r_n = \infty$, and let $\nu_n$, $n\in\N$, be the probability measure defined by
\beq\label{nun}
\int_{Y_d} f(y)\, d\nu_n(y) = \frac{1}{r_n} \int_0^{r_n} f(X(s,y_n))\,ds\quad\mbox{for }f\in C^0_\sharp(Y_d).
\eeq
Then, there exists a subsequence $(\nu_{n_k})_{k\in\N}$ of $(\nu_n)_{n\in\N}$ which converges weakly~$*$ to some probability measure $\mu\in \cM_p(Y_d)$ which is invariant for the flow $X$. 
\end{alem}
\par
Let $a$ be a limit point of the sequence $(u_n)_{n\in\N}$ \eqref{ung}, namely 
\[
a = \lim_{n\to\infty}\,\frac{1}{t_{\theta(n)}}\int_0^{t_{\theta(n)}} g(X(t,x_{\theta(n)}))\,dt,
\]
for some strictly increasing sequence $(\theta(n))_{n\in\N}$ of integer numbers. 
Set $r_n := t_{\theta(n)}$, $y_n:= x_{\theta(n)}$, and consider the associated sequence $(\nu_n)_{n\in\N}$ of probability measures on $Y_d$ given by \eqref{nun}. 
By Lemma~\ref{lem-extrac} we can extract a subsequence $(\nu_{n_k})_{k\in\N}$ which converges weakly~$*$ to some invariant probability measure $\mu\in \cI_b$ for the flow $X$.
We thus have
\[
\forall\,f\in C^0_\sharp(Y_d),\quad \lim_{k\to\infty} \int_{Y_d} f(y)\, d\nu_{n_k}(y) = \int_{Y_d} f(y)\, d\mu(y),
\]
which implies in particular that
\[
a = \lim_{k\to\infty}\frac{1}{t_{\theta(n_k)}}\int_0^{t_{\theta(n_k)}} g(X(s,x_{\theta(n_k)}))\,ds
= \lim_{k\to\infty} \int_{Y_d} g(y)\, d\nu_{n_k}(y) = \int_{Y_d} g(y)\, d\mu(y).
\]
\end{proof}
\noindent
\begin{proof}{ of Lemma~\ref{lem-extrac}}
Since $Y_d$ is a compact metrizable space, there exists a subsequence $(\nu_{n_k})_{k\in\N}$ of $(\nu_n)_{n\in\N}$ which converges weakly~$*$ to some probability measure $\mu\in \cM_p(Y_d)$, namely for any $f\in C^0_\sharp(Y_d)$,
\beq\label{wcnunk}
\int_{Y_d} f(y)\, d\nu_{n_k}(y) = \frac{1}{r_{n_k}} \int_0^{r_{n_k}} f(X(s,y_{n_k}))\,ds\;\mathop{\longrightarrow}_{k\to\infty}\;\int_{Y_d} f(y)\, d\mu(y).
\eeq
Let us prove that $\mu$ is invariant for the flow $X$. For the sake of simplicity denote $\tau_k:= r_{n_k}$, $z_k:=y_{n_k}$ and $\mu_k:= \nu_{n_k}$.
Let $t\in\R$ and $f\in C^0_\sharp(Y_d)$.
By the semi-group property of the flow \eqref{sgroup} we have
\[
\int_{Y_d} (T_t f)(y)\, d\mu_k(y) =  \frac{1}{\tau_k} \int_0^{\tau_k} f(X(s+t,z_k))\,ds.
\]
By the change of variable $r=s+t$, it follows that 
\[
\ba{l}
\dis \int_{Y_d} (T_t f)(y)\, d\mu_k(y) =\frac{1}{\tau_k} \int_{t}^{t+\tau_k} f(X(r,z_k))\,dr
\\ \ecart
\dis =\frac{1}{\tau_k} \int_{0}^{\tau_k} f(X(r,z_k))\,dr + \frac{1}{\tau_k} \int_{\tau_k}^{t+\tau_k} f(X(r,z_k))\,dr- \frac{1}{\tau_k} \int_{0}^{t} f(X(r,z_k))\,dr
\ea
\]
Since $f$ is bounded and $t\in\R$ is fixed, we deduce from \eqref{wcnunk} that
\[
\lim_{k\to \infty} \int_{Y_d} (T_t f)(y)\, d\mu_k(y) = \int_{Y_d} f(y)\, d\mu(y).
\]
However, by the definition of $\mu$ we also have
\[
\lim_{k\to \infty} \int_{Y_d} (T_t f)(y)\, d\mu_k(y) = \int_{Y_d} (T_t f)(y)\, d\mu(y).
\]
Hence, we get that
\[
\forall\,t\in\R,\ \forall\,f\in C^0_\sharp(Y_d),\quad \int_{Y_d} (T_t f)(y)\, d\mu(y) = \int_{Y_d} f(y)\, d\mu(y),
\]
which implies that $\mu$ is invariant for the flow $X$.
\end{proof}
%%%%%%%%%%

%%%%%%%%%%
\end{document}